\documentclass{amsart}
\usepackage{bbm}
\usepackage{bbding}
\usepackage[mathcal]{euscript}
\usepackage{graphicx}
\usepackage{amsthm}
\usepackage{calc}
\usepackage{mathrsfs,dsfont}
\usepackage{CJK,fancyhdr,amscd}
\usepackage{anysize} 
\usepackage{amsmath,amssymb,amsfonts}
\usepackage{epsfig,enumerate}
\usepackage{latexsym}
\usepackage{indentfirst, latexsym}
\usepackage{graphics}
\usepackage[all,poly,knot]{xy}
\usepackage[colorlinks,linkcolor=blue,anchorcolor=blue,citecolor=blue,pagebackref]{hyperref}
\usepackage{geometry}
\numberwithin{equation}{section}
\newtheorem{theorem} {Theorem} [section]
\newtheorem{proposition}[theorem]{Proposition}
\newtheorem{corollary}  [theorem]     {Corollary}
\newtheorem{lemma}  [theorem]     {Lemma}
\newtheorem{question}  [theorem]     {Question}

\newtheorem{remark}  [theorem]     {Remark}

\newtheorem{conjecture}  [theorem]     {Conjecture}

\theoremstyle{definition}

\newtheorem{example}  [theorem]     {Example}

\newcommand{\Rm}{\operatorname {R}}
\newcommand{\ad}{\operatorname{ad}}
\newcommand{\Ric}{\operatorname{Ric}}

\renewcommand*\backref[1]{}
\renewcommand*\backrefalt[4]{ \ifcase #1 \or (cited on page #2) \else (cited on pages #2) \fi}

\begin{document}

\title{A classification of locally Chern homogeneous Hermitian manifolds}

\author{Lei Ni}
\address{Lei Ni. Department of Mathematics, University of California, San Diego, La Jolla, CA 92093, USA}
\email{leni@ucsd.edu}

\author{Fangyang Zheng} \thanks{Zheng is the corresponding author. The research is partially supported by NSFC grants \# 12071050  and 12141101, Chongqing grant cstc2021ycjh-bgzxm0139 and Chongqing Normal University grant 19XRC001, and is supported by the 111 Project D21024.}
\address{Fangyang Zheng. School of Mathematical Sciences, Chongqing Normal University, Chongqing 401331, China}
\email{20190045@cqnu.edu.cn; \ franciszheng@yahoo.com }

\subjclass[2020]{53C55 (primary), 53C05 (secondary)}
\keywords{Hermitian manifold; Chern connection; Ambrose-Singer connections; Hermitian (locally) homogeneous manifolds}

\begin{abstract}
We apply the algebraic consideration of holonomy systems to study  Hermitian manifolds whose Chern connection is Ambrose-Singer and prove structure theorems for such manifolds. The main result (Theorem 1.2) asserts that the universal cover of such a Hermitian manifold must be the product of a complex Lie group and Hermitian symmetric spaces, which was previously proved up to complex dimension four by the authors.  This in some sense is the Hermitian version of Cartan's classification of Hermitian symmetric spaces. We also obtain  results on Hermitian manifolds whose Bismut connection is Ambrose-Singer when the complex dimension $n\le 4$. Furthermore we discuss a project of classifying such manifolds via Alekseevski\u{i} and Kimel\!\'{}\!fel\!\'{}\!d type theorems, which we establish for the family of the Gauduchon connections on a compact Hermitian manifold except for the Bismut connection.
\end{abstract}

\subjclass[2020]{ 53C55 (Primary), 53C05 (Secondary)}
\keywords{Ambrose-Singer connection, Chern connection, holonomy algebra/group/system, Chern flat, locally homogeneous Hermitian manifolds}

\maketitle

\markleft{Ni and Zheng}
\markright{locally Chern homogeneous Hermitian manifolds}

\tableofcontents

\section{Introduction and main results}

In \cite{AS} Ambrose and Singer  gave a necessary and sufficient condition for a simply-connected complete Riemannian manifold to admit a transitive group of isometric motions. These conditions were recognized later (\cite{Kostant-Nagoya}) as the existence of an affine connection which is {\it invariant under parallelism,} that is, the connection has parallel torsion and  curvature with respect to itself. Namely

\begin{theorem}[Ambrose-Singer, Kostant] Let $(M, g)$ be a complete, simply-connected Riemannian manifold. Then $M$ admits a transitive action of an isometric group (namely $M$ is a Riemannian homogeneous space)  if and only if there exists a metric connection $D$ such that its torsion and curvature are parallel with respect to $D$.
\end{theorem}

The conditions were originally expressed in terms of  more convoluted PDEs involving the curvature tensor of the Levi-Civita connection $\nabla^g$ (also denoted as $\nabla^{LC}$) and the torsion of the metric connection $D$ ($D\ne \nabla^g$ unless $(M, g)$ is a symmetric space)  as a characterization of simply-connected complete Riemannian homogeneous spaces (cf.\,\cite{AS}). The above formulation factored in the results of \cite{Kostant-Nagoya} (Theorem 2 and Lemma 3). This theorem provided  a generalization of Cartan's results  on the symmetric (or locally symmetric) spaces.\footnote{In a private conversation, R. Hamilton suggested that the convergence of a Ricci flow solution to a homogeneous structure perhaps can be captured by the emergence of an Ambrose-Singer connection.}
Without the simply-connectedness the result holds locally. The metric connection $D$ is called an {\em Ambrose-Singer connection} (abbreviated as an AS connection).

 It is well known that on a Hermitian manifold there exists a unique canonical connection, namely the Chern connection \cite{Chern} (in fact the existence, as well as the uniqueness,  holds for any holomorphic vector bundle over a complex manifold). In this paper we first  study complete Hermitian manifolds that are {\it locally Chern homogeneous,} namely, its Chern connection $\nabla$ satisfies  the condition $\nabla T=0$ and $\nabla R=0$, where $T$ and $R$ are respectively the torsion and curvature of $\nabla$. This notion is motivated by the above work of Ambrose-Singer \cite{AS} and others such as \cite{Kob, Nomizu, Kostant-Nagoya} (see also \cite{Ballmann}). In particular, there is a corresponding result of Sekigawa \cite{Seki} for Hermitian manifolds which asserts that a complete simply-connected Hermitian manifold is Hermitian homogeneous (namely Hermitian manifolds whose biholomorphic isometries act transitively) if and only if there exists a Hermitian Ambrose-Singer connection, namely in addition to the assumptions of having parallel torsion and curvature,  the metric connection also preserves the almost complex structure.

We  refer the readers to \cite{Kostant, Singer, TV, Wang, Wang-PAMS} for other related studies on homogeneous complex manifolds (namely complex manifolds whose biholomorphism groups act transitively)  and other  homogeneous spaces.
One should also consult \cite{Ang-Ped} for some recent developments on local Hermitian homogeneous spaces and the flows on such spaces. Locally Chern homogeneous Hermitian manifolds form a special class of locally homogenous complex manifolds as shown by the theorem below.

  For a locally Chern homogeneous manifold $(M, g)$, namely when the Chern connection of a Hermitian manifold $(M^n, g)$ is Ambrose-Singer, we also call  the complex manifold $M$ admits a {\em CAS structure}. In our earlier work \cite[Theorem 3.6]{NZ-CAS} conditions were given for the universal cover of such a manifold to admit some K\"ahler (hence  Hermitian symmetric due to $\nabla R=0$) de Rham factors, and classification  theorems were obtained in complex dimension $3$ and $4$ under the additional compactness assumption of the manifold.   The examples of manifolds with CAS structure include complex Lie groups endowed with compatible  left-invariant metrics and Hermitian symmetric spaces, their products and quotients of products. In this paper, we prove that the reverse holds by the following classification result for locally Chern homogeneous manifolds in general dimensions, which is the first main result of this article:

 \begin{theorem}\label{thm:Class} Let $(M^n, g)$ be a complete Hermitian manifold with a CAS structure, namely a locally Chern homogeneous Hermitian manifold. Then its universal cover $\widetilde{M}$ splits into $M_1\times M_2$, where $M_1$ is a Chern flat Hermitian manifold with a complex Lie group structure, and $M_2$ is the product of irreducible Hermitian symmetric spaces of dimension $k$ with $0\leq k\leq n$.
\end{theorem}

In \cite{NZ-CAS}, it was shown that the above theorem holds when $n\leq 4$ for the case that $M$ is compact. The proof of the above theorem consists of three parts. The first part is to split off the K\"ahler de Rham factors (if any) in the universal cover of $M$, which correspond to the kernel distribution of the Chern torsion. Note that  since the connection is {\em not} Levi-Civita (when the metric is not K\"ahler), the de Rham decomposition theorem no longer holds to its full generality, but the CAS assumption will enable us to get the desired splitting. The second part is to show that a CAS manifold without K\"ahler de Rham factors must be (first) Chern Ricci flat. This is achieved by constructing a holomorphic symplectic form on the manifold. These two parts were obtained in
\cite{NZ-CAS}, and we will recall and give their outline in the next section for readers' convenience.  The third part is an algebraic analogue to the classic theorem of Alekseevski\u{i} and Kimel\!\'{}\!fel\!\'{}\!d \cite{AK}, where we show that for an abstract holonomy system tailored to our CAS situation, if it is Ricci flat then it is flat. While the first two steps are quite different from the de Rham theorem for the Symmetric spaces, the third step is close to Cartan's work on the symmetric spaces (Hermitian symmetric spaces) via the Riemannian symmetric Lie algebras (Hermitian Lie algebras) and the holonomy groups \cite{Cartan-27} (cf. Simons' proof \cite{Sim} of Berger's holonomy theorem). Here we also consider a generalized holonomy system for connections with  non-vanishing torsion. Although results of Ambrose-Singer and Segigawa together with the work of \cite{Nomizu, Kostant-Nagoya} have indicated that the CAS manifolds admit many locally infinitesimal affine transformations, a completely classification as Theorem \ref{thm:Class} still comes as a surprise.

Recall that Alekseevski\u{i} and Kimel\!\'{}\!fel\!\'{}\!d proved in \cite{AK} that any Ricci flat complete homogeneous Riemannian  manifold $M^n$ must be flat, and in fact it is isometric to $\mbox{T}^k\times \mathbb{R}^{n-k}$, where $\mbox{T}^k$ is a flat torus. They proved this result by the consideration of a volume entropy. It can also be derived from Cheeger-Gromoll's splitting theorem and the consideration of Clifford translation \cite{Wolf}. We will abbreviate a statement deducing the flatness from the Ricci flatness as an {\em AK type theorem} in this paper. Since any Ricci flat Hermitian symmetric space must be flat, our Theorem \ref{thm:Class} has the following immediate consequence which can be viewed as a Hermitian analogue to the aforementioned AK theorem:

\begin{corollary}\label{coro:AK} Let $(M^n, g)$ be a complete Hermitian manifold with a CAS structure. Assume that the Chern curvature is Ricci flat. Then it is Chern flat.  In particular, $M$ is covered by a complex Lie group.
\end{corollary}

Note that one no longer has the Cheeger-Gromoll splitting theorem for Hermitian manifolds with Chern connection. Our proof to the above result is algebraic instead. The result answers in the case of CAS manifolds a question raised by Yau \cite[Problem 87]{Yau} where he asked if {\it  one can say something nontrivial about the compact Hermitian manifolds whose holonomy is a proper subgroup of $\mathsf{U}(n)$}, since the Ricci flatness result of the previous paper \cite{NZ-CAS} asserts that the holonomy of a CAS manifold is a proper subgroup of $\mathsf{SU}(n)$. The above two results are proved in \S 4 after some preliminaries in \S 2 and \S3.

Recall that for any given metric connection $D$ on a Riemannian manifold, its curvature tensor $R^D$ is a covariant $4$-tensor which is skew-symmetric with respect to its first two or last two positions. Its Ricci curvature $\Ric^D$ is the $2$-tensor obtained by contracting the first and fourth (or equivalently, its second and third) positions of $R^D$. When the manifold is Hermitian and $D$ is also a Hermitian connection, then there are three ways to contract its curvature tensor $R^D$, resulting in the {\em first, second,} and {\em third} Ricci of $D$, with $\Ric^D$ being the third Ricci. More precisely,  recall that the Ricci curvature tensor of a metric connection $D$ on a Riemannian manifold $(M,g)$ is defined by
\begin{equation} \label{eq:Ricci}
 \mbox{Ric}^D(x,y) = \sum_{i=1}^m R^D(\varepsilon_i, x, y, \varepsilon_i)
 \end{equation}
where $R^D(x,y,z,w)= \langle D_xD_yz-D_yD_xz-D_{[x,y]}z, w\rangle$ is the curvature of $D$,  $\{ \varepsilon_i \}$ is a local orthonormal frame of $M$, while $x, y, z, w$ are tangent vectors. In general $ \mbox{Ric}^D$ might not be symmetric as $D$ may have torsion. When $(M^n,g)$ is Hermitian, we may choose the orthonormal frame $\{ \varepsilon_1, \ldots \varepsilon_{2n}\}$ so that
$$ \sqrt{2}\varepsilon_i = e_i+\overline{e}_i, \ \ \ \sqrt{2}\varepsilon_{n+i} = \sqrt{\!-\!1}(e_i-\overline{e}_i) ; \ \ \ \ \  i=1, \ldots , n,  $$
where $\{ e_1, \ldots , e_n\}$ is a local unitary frame. Plugging into (\ref{eq:Ricci}), we get
\begin{equation} \label{eq:Ricci2}
 \mbox{Ric}^D(x,y) = \sum_{i=1}^n \{ R^D(e_i, x, y, \overline{e}_i) + R^D(\overline{e}_i, x, y, e_i)\}.
 \end{equation}
Now suppose $D$ is Hermitian, namely, $Dg=0$ and $DJ=0$. Then $R^D(\ast , \ast ,e_i, e_j)=0$. So for any type $(1,0)$ tangent vectors $X$, $Y$ we have
\begin{equation} \label{eq:Ricci3}
 \mbox{Ric}^D(X,Y) = \sum_{i=1}^n  R^D(e_i, X, Y, \overline{e}_i) , \ \ \ \mbox{Ric}^D(\overline{X},Y)= \sum_{i=1}^n  R^D(e_i, \overline{X}, Y, \overline{e}_i).
 \end{equation}
This is often called the {\em third Ricci curvature} of the Hermitian connection $ D$, also denoted as $\Ric^{D(3)}$, while the {\em first} and {\em second} Ricci are defined respectively by
$$ \mbox{Ric}^{D(1)}(x,y)=\sum_{i=1}^n R^D(x,y,e_i, \overline{e}_i), \ \ \  \ \ \  \mbox{Ric}^{D(2)}(X, \overline{Y})=\sum_{i=1}^n R^D(e_i, \overline{e}_i, X, \overline{Y}).$$

 For Chern connection the only non-vanishing part of $ \mbox{Ric}^{D(1)}(\cdot, \cdot)$ are $\Ric(X, \overline{Y})$ and their conjugates. For CAS manifolds, the Chern curvature tensor obeys all K\"ahler symmetries, namely CAS manifolds are {\em Chern K\"ahler-like} (cf. Lemma 3.1 of \cite{NZ-CAS}).  Hence all three Chern Ricci curvatures coincide. For general Hermitian manifolds, these three Ricci curvatures may not be equal even for the Chern connection.

\begin{remark} \label{remark14}
We remark that the naive way of extending AK theorem from Levi-Civita connection to Chern connection fails, namely, there are (locally) Hermitian homogeneous manifolds with vanishing (third) Chern Ricci curvature but are not Chern flat. (For instance, some (compact quotients of) {\em almost abelian} Lie groups will be (third) Chern Ricci flat but not Chern flat. There are also examples of compact locally homogeneous Hermitian manifolds with vanishing first and third Chern Ricci yet is not Chern flat. Such examples are given at the end of \S 4). From this, we see that the CAS assumption in Corollary \ref{coro:AK}, which is stronger than just assuming the manifold to be (locally) Hermitian homogeneous, is an appropriate one for an AK type theorem for the Chern connection.

We remark that AK type theorems were studied previously by \cite{GFS}, \cite{Podesta-R},  etc, for other connections.

\end{remark}

Besides the Chern connection $\nabla$, another important canonical connection on a Hermitian manifold $(M^n,g)$ is the {\em Bismut connection} $\nabla^b$, which is the unique Hermitian connection (meaning a metric connection which makes the almost complex structure $J$ parallel) whose torsion is totally skew-symmetric.  The Bismut connection is extensively studied by physicists in string theory, and is an essential component in the Hull-Strominger system which suggests that the hidden space is a non-K\"ahler Calabi-Yau manifold (cf. \cite{GHR}, \cite{Strominger}). In mathematics, Bismut \cite{Bismut} established the existence and uniqueness of this connection and used it in his study of local index theory. It is also related to the Hermitian Ricci flow theory (cf. \cite{Streets}, \cite{StreetsTian}, \cite{StreetsTian1} and references therein). It seems that the geometric and complex analytic meaning of Bismut curvature is less understood when compared with Chern curvature. Moreover, Bismut curvature does not seem to enjoy all the symmetries satisfied by the Chern curvature. There is an interesting geometric explanation for this in the frame work of generalized Ricci flow developed by Garcia-Fern\'andez and Streets \cite{GFS} (see also the references therein).

Analogous to the Chern connection case, we consider {\em Bismut Ambrose-Singer} manifolds (BAS for brevity), namely  Hermitian manifolds whose Bismut connection has parallel torsion and curvature. However, Theorem \ref{thm:Class} fails completely when the Chern connection is replaced by the Bismut connection since there are plenty of examples of BAS manifolds that are `non-trivial', in the sense that they are neither Hermitian symmetric nor Bismut flat. This can be seen by results in \cite{AndV, Ang-Ped}. By \cite{Ang-Ped}, the Vaisman sovmanifolds are BAS, while Theorem 4.7 of \cite{AndV} asserts that  Vaisman solvmanifolds have nontrivial holonomy with respect to the Bismut connection.    Below is  an example by straight forward calculations.

\begin{example} [{\bf Hopf manifold with $\textbf{n}\geq \textbf{3}$}]

Let $M^n=({\mathbb C}^n\setminus \{ 0\})/\langle f\rangle$, where  $z=(z_1, \ldots , z_n)$ and $f(z)=2z$. Let us write $|z|^2=|z_1|^2+\cdots + |z_n|^2$, $\omega = \sqrt{-1}\frac{\partial\overline{\partial}|z|^2}{|z|^2}$, and let $g$ be the Hermitian metric with K\"ahler form $\omega$. As is well-known, the locally conformally K\"ahler metric $g$ is not Bismut flat when $n\geq 3$, but is always Vaisman, meaning that its Lee form is parallel under the Levi-Civita connection $\nabla^g$. By  Corollary 3.8 of \cite{AndV} we know that the Bismut connection of $g$ has parallel torsion (BTP). In the appendix (\S 8), we give the explicit computation showing that the Bismut curvature of the Hopf manifold $(M^n,g)$ is also parallel,  and it is not Bismut flat (when $n\geq 3$), so it gives an example of BAS manifold which is neither Hermitian symmetric nor Bismut flat.
\end{example}

As a matter of fact, starting in dimension $3$, there are plenty of other examples of BAS manifolds that are not K\"ahler (hence not locally Hermitian symmetric) and not Bismut flat, and there are even examples of BAS manifolds that are balanced (but non-K\"ahler) (compare with Theorem \ref{prop1.7a} below).

Note that BAS manifolds form a subset of (locally) {\em naturally reductive homogeneous spaces,} which means homogeneous Riemannian manifolds admitting a metric connection with parallel (with respect to itself) and totally skew-symmetric torsion (cf. Chapter 6 of \cite{TV}). Such spaces were systematically studied by Agricola, Friedrich and others (see \cite{Agricola}, \cite{AgricolaF}, and the references therein). Classification of real dimensions up to $8$ were obtained (see \cite{AgricolaFF}, \cite{KV}, \cite{Storm} and the references therein). While it might be difficult  to classify naturally reductive spaces in general dimensions, it seems to be a plausible and  interesting problem to  classify or characterize all BAS manifolds. As a step towards this goal we ask whether or not the analogue to Corollary \ref{coro:AK}  holds for the Bismut connection, namely, for a compact BAS manifold with vanishing (third) Bismut Ricci, must it be Bismut flat? It is the case for $n=2$ and  we suspect that the answer would be negative when $n\geq 3$, although we have not been able to construct a concrete counterexample yet. On the positive side, we believe the following should be true:

\begin{conjecture} \label{conj1.6a}
Let $(M^n,g)$ be a compact (or complete) Hermitian manifold such that its Bismut connection $\nabla^b$ is Ambrose-Singer. If the first and third Bismut Ricci both vanish, then the Bismut curvature vanishes.
\end{conjecture}

Note that  for BAS manifolds, the first and second Bismut Ricci curvatures are  always equal (for example by equation (1.2) of Theorem 1.1 in \cite{ZZ22}), while they are not necessarily equal to the third. Moreover similar to Chern connection, for a BAS manifold,  the only non-vanishing part of the first Bismut Ricci is $\Ric (X, \overline{Y})$ and its conjugate.   Recall also  that the compact Bismut flat manifolds have been fully classified \cite{WYZ} as quotients of  Samelson spaces. Hence the above conjecture, if confirmed, would provide a classification for Ricci flat compact BAS manifolds as an extension of Corollary \ref{coro:AK} to the Bismut connection.

Regarding to Conjecture \ref{conj1.6a}, we show in \S 6 that it holds up to complex dimension four:
\begin{theorem} \label{prop1.7a}
Let $(M^n,g)$ be a complete, non-K\"ahler BAS manifold with vanishing first and third Bismut Ricci. Then it is not balanced. Furthermore, it is Bismut flat   if $n\leq 4$.
\end{theorem}
Another supporting evidence to the conjecture can be obtained for {\em Bismut K\"ahler-like} metrics, namely Hermitian metrics whose Bismut curvature tensor obeys all K\"ahler symmetries. In this case all three Bismut Ricci curvatures coincide. The following is a weaker version of \cite[Theorem 3]{ZZ23}, and we give an alternative proof here as an application of our algebraic result on holonomy systems in \S 3 and \S 5:
\begin{proposition}\label{prop1.8a}
Let $(M^n,g)$ be a complete BAS manifold with vanishing first Bismut Ricci. Assume that it is Bismut K\"ahler-like. Then it must be Bismut flat.
\end{proposition}

Besides Chern and Bismut connections, the family of Gauduchon connections, which is the line joining Chern and Bismut, have also been studied in the literature. For any real number $t$, the {\em $t$-Gauduchon connection} is defined by $\nabla^{(t)}=(1-\frac{t}{2})\nabla + \frac{t}{2}\nabla^b$. So in particular, $\nabla^{(0)}=\nabla$ is the Chern connection, $\nabla^{(2)}=\nabla^b$ is the Bismut connection, while $\nabla^{(1)}$ is the Hermitian projection of the Levi-Civita connection, which is called the {\em Lichnerowicz connection} in some literature. Note that when the metric $g$ is K\"ahler, all  $\nabla^{(t)}$ coincide with the Levi-Civita connection $\nabla^g$. But when $g$ is not K\"ahler, $\nabla^{(t)}\neq \nabla^{(t')}$ whenever $t\neq t'$, and none of them can be equal to $\nabla^g$ as the latter is not a Hermitian connection in this case.  Lafuente and Stanfield \cite{LS} proved the following beautiful theorem about Gauduchon flat manifolds:

\begin{theorem}[Lafuente and Stanfield \cite{LS}]
For any $t\neq 0,2$, if a compact Hermitian manifold has flat $t$-Gauduchon connection $\nabla^{(t)}$, then it must be K\"ahler.
\end{theorem}

In fact what they proved is  stronger: the result holds when the flatness is replaced by the K\"ahler-like condition, namely, when the curvature of $\nabla^{(t)}$ obeys all K\"ahler symmetries (see for example \cite{ZZ-GKL} and references therein for more discussion on $t$-Gauduchon K\"ahler-like manifolds). Lafuente and Stanfield also showed that the metric must be K\"ahler when the compactness assumption is dropped but when $t$ is not in  $\{ \frac{2}{3}, \frac{4}{5} \}$.

For  any $t\neq 0,2$, similar to the Chern or Bismut case, one could ask the question about $t$-GAS manifolds, namely those with $\nabla^{(t)}$ being Ambrose-Singer. It turns out that  there are examples of compact $t$-GAS manifolds that are `non-trivial', in the sense that they are non-K\"ahler (thus can be neither Hermitian symmetric nor $t$-Gauduchon flat). For instance, take any compact quotient of a simple complex Lie group, equipped with the standard metric coming from the Cartan-Killing form (see \S 7 for more details), then it is $t$-GAS for any $t\in {\mathbb R}$. It seems that $t$-GAS manifolds form a highly restrictive class, so a full classification of them constitutes an interesting project for a future study.

As a reduction one could also ask the AK type question for $t$-GAS manifolds when $t\neq 0,2$, namely, {\em if $\nabla^{(t)}$ has vanishing (third) Ricci curvature, then must $\nabla^{(t)}$ be flat?} We suspect that the answer would be a YES. In this direction we  prove in \S 7 the following

\begin{theorem} \label{prop1.10}
Let $t\neq 0,2$ and $(M^n,g)$ be a compact Hermitian manifold such that its $t$-Gauduchon connection $\nabla^{(t)}$ is Ambrose-Singer.  Then $g$ must be balanced. Furthermore,  if $\nabla^{(t)}$  has vanishing first and third Ricci curvatures, then $\nabla^{(t)}$ must be K\"ahler and flat, hence a (finite undercover of a) complex torus.
\end{theorem}
In fact the result holds under weaker assumption that the first and third scalar curvatures vanish (cf. Corollary \ref{coro:71}). In particular, together with Corollary \ref{coro:AK}, the above implies that the AK type theorem can be established for Gauduchon family as long as $t\neq  2$ if the Ricci flatness is assumed for the first and third Ricci. Hence {\it a negative answer to  Conjecture \ref{conj1.6a}  would make the Bismut connection a special exception.}

Note that in the aforementioned AK type question for $t$-GAS manifolds, one can not replace the flatness assumption on the first and third Ricci by the first Ricci alone, as there are examples of non-K\"ahler $t$-GAS manifolds which has vanishing first $t$-Gauduchon Ricci, as pointed out by  Lafuente and Stanfield in \cite[page 2]{LS} (see also \cite{Vez}), on a balanced Hermitian nimanifold $M$ with vanishing first Chern Ricci form. On such a manifold,  $\nabla^{(t)}$ has vanishing first Ricci for all $t$, since the difference of the first Ricci forms between two Gauduchon connections is given by the differential of Gauduchon's torsion $1$-form $\eta$, while the balanced condition implie thats $\eta=0$. Hence this example supplies a counter example to the above theorem if one drops the assumption  of the third Ricci curvature being  vanishing.

\vspace{0.3cm}

\section{Preliminaries}


In this section we will collect some results from \cite{NZ-CAS} which will be used in the proof of our main result, Theorem \ref{thm:Class} stated in the introduction.

Let $(M^n,g)$ be a complete Hermitian manifold. Denote by $\nabla$ the Chern connection, by $T$ and $R$ its torsion and curvature. The following was proved in \cite{NZ-CAS} (cf. also \cite{YZ}):

\begin{proposition}\label{prop:21} Suppose that $(M^n,g)$ is a Hermitian manifold with a CAS structure. Then the Chern curvature  is  K\"ahler-like. In particular, the Chern curvature satisfies the first Bianchi identity.
\end{proposition}

In \cite{NZ-CAS}, we denoted by $\mathcal{W}_p$  the subspace of $T^{1, 0}_p$ spanned by the image of $T$:
 $$\mathcal{W}:= \langle \{ T(X,Y)\, | \,  X, Y \in T^{1,0}M\}\rangle.$$ Let  $ \mathcal{N}:=\{ Z \in T^{1, 0}M \mid \langle T(X,Y), \overline{Z} \rangle =0 \ \ \forall \ X, Y \in T^{1,0}M\}.$ The following was proved in \cite{NZ-CAS}.

 \begin{proposition} The subbundles $\mathcal{W}$ and $\mathcal{N}$ of $T^{1, 0}M$ (also denoted as $T'M$) are invariant under the parallel transport with respect to the Chern connection $\nabla$, and $T'M=\mathcal{W}\oplus \mathcal{N}$ orthogonally. The curvature restriction $R|_{\mathcal{W}}=0$. In particular, when $M$ is not a locally Hermitian symmetric space, the action of the holonomy group $G$ on $T'M$ is reducible.
 \end{proposition}

The next result extends the construction, mainly Theorem 1.4 and its proof in Section 6 of \cite{NZ-CAS}. Note that there was a typo in the statement of \cite[Theorem 1.4]{NZ-CAS}, namely, the word ``with $k\geq 2$" at the end should be deleted. We will recap the proof here for the sake of completeness.

\begin{theorem} Let $(M^n, g)$ be a complete simply-connected Hermitian manifold with a CAS structure. Assume that it does not admit any K\"ahler de Rham factor. Then $\mathcal{N}$ decomposes into $\mathcal{N}_1\oplus \mathcal{N}_2$ with $R|_{\mathcal{W}+\mathcal{N}_2}=0$ and $\mathcal{N}_1$ decomposes further into $\oplus_{j=1}^k \mathcal{K}_j$, with each of $\mathcal{K}_j$ being invariant and irreducible under the action of the holonomy group $G$. On each $\mathcal{K}_j$ there exists a parallel $(2,0)$-symplectic form which can be identified with the standard holomorphic symplectic $(2,0)$ form on $\mathbb{C}^{2k_j}$ with $2k_j=\dim_{\mathbb{C}}(\mathcal{K}_j)$.
\end{theorem}
\begin{proof} We first recap the constructions in the proof of Theorem 1.4 in \cite{NZ-CAS}. By the simply-connectedness assumption and the Ambrose-Singer holonomy theorem, namely Theorem  \ref{thm:AS} stated in the next section, $\mathcal{W}$ is a trivial bundle which admits parallel (holomorphic) sections $Z_1, \cdots, Z_{\ell_1}$ where $\ell_1=\dim(\mathcal{W}_p)$. Let $\xi_i=\langle \cdot, \overline{Z}_{i}\rangle$ be the dual $1$-forms correspondingly. Then let  $\tau_i(\cdot, \cdot)=d\xi_i=\partial \xi_i$. It was proved in Lemma 6.2 of \cite{NZ-CAS} that $\tau_i(\cdot, \cdot)=\langle T(\cdot, \cdot), \overline{Z}_i\rangle$, which also equals to $-\iota_{\overline{Z}_i}\partial \omega$.  It was proved there, also can be seen by direct calculation, that $\tau_i(\cdot, \cdot)=\partial \xi_i$ and $\bar{\partial}\xi_i=0$,  that $\{\tau_i\}$ are parallel, holomorphic $(2, 0)$-forms.

The next key step in the proof of Theorem 1.4 of \cite{NZ-CAS} is to pick one $\tau$ in $\operatorname{span}\{\tau_i\}$ such that $\tau|_{\mathcal{N}}$ is of maximum rank. Without the loss of generality we denote this one by $\tau_1$. It then derived from the assumption that $M$ has no K\"ahler factor and $\tau_1$ is of maximum rank, there exists a orthogonal decomposition of $\mathcal{N}$ into $\mathcal{N}_1\oplus\mathcal{N}_2$ such that $\mathcal{N}_2$ is generated by parallel global sections, hence $R|_{\mathcal{N}_2}=0$. Moreover $\mathcal{N}_1$ is of even dimension and $\tau_1|_{\mathcal{N}_1}$ is non-degenerate. Let $\ell_2=\dim_{\mathbb{C}}(\mathcal{N}_2)$. Then $\dim_{\mathbb{C}}(\mathcal{N}_1)=2k$ and $2k+\ell_1+\ell_2=n$. Since $\mathcal{W}$ and $\mathcal{N}_2$ are generated by parallel (holomorphic) sections/vector fields, both $\mathcal{K}_0:=\mathcal{W}\oplus \mathcal{N}_2$ and $\mathcal{N}_1$ are invariant under the parallel transport with respect to the Chern connection. Fix a point $p\in M$ the holonomy action of any element $h\in G$ splits as $\tilde{h}\oplus \operatorname{id}: (\mathcal{N}_1)_p\oplus (\mathcal{K}_0)_p\to (\mathcal{N}_1)_p\oplus (\mathcal{K}_0)_p$.

Now consider $\tau|_{\mathcal{N}_1}$, a parallel $(2, 0)$ form on this sub-bundle of $T'M$, which we shall still denote by $\tau$. With respect to a chosen unitary frame of $\mathcal{N}_1$ at the point $p$, $\tau$ has the following matrix form:
$$A=\left[ \begin{array}{llll} a_1E & & & \\ & \ddots & & \\  & & a_kE &    \end{array} \right] ,   \ \ \ \ \ E = \left[ \begin{array}{ll } 0 & 1 \\ -1 & 0 \end{array} \right], \ \ \ \ a_1\geq \cdots \geq a_k >0.$$
We may collect the $a_i$ into distinct numbers of $b_j$. Namely  we write
$$
b_1\widetilde{E}=\left[ \begin{array}{llll} a_1E & & & \\ & \ddots & & \\  & & a_1E &    \end{array} \right],\quad  A=\left[ \begin{array}{llll} b_1\widetilde{E} & & & \\ & \ddots & & \\  & & b_{\ell_3}\widetilde{E} &    \end{array} \right] ,   \ \ \  b_1> \cdots > b_{\ell_3} >0.
$$
Namely $-b^2_j$ are distinct eigenvalues of $A \overline{A}$ if $A$ is the matrix representation of $\tau$ at $p$. Let $2n_j$ denote the dimension of the corresponding eigenspaces $\mathcal{V}_j$. Clearly $\sum_{j=1}^{\ell_3} n_j=k$. Since $\tau$ is parallel, $\tilde{h}^*\tau=\tau$ for any $h\in G$. In terms of matrix we have that $\tilde{h}^{tr} A \tilde{h} =A$. Here $\tilde{h}^{tr}$ denotes the transpose of $\tilde{h}$. Note that $\tilde{h}\in U(2k)$, $A\tilde{h}=\bar{\tilde{h}} A$. Since $A$ is real we have $\tilde{h}A=A \bar{\tilde{h}}$. From this we have that $\tilde{h}A^2=A^2 \tilde{h}$, which implies that $\tilde{h}$ also keeps each eigenspace $\mathcal{V}_j$ invariant.  This then implies that $\mathcal{N}_1$ further decomposes into orthogonal holonomy-invariant subbundles. The next statement follows from Weyl's completely reducibility theorem. Thus $\tau$ is a constant  multiple of the standard holomorphic symplectic $(2,0)$-form on each $\mathcal{V}_j$.
\end{proof}

\begin{corollary}\label{coro:21} For any $x, y\in T_pM$, $(R_{x, y})|_{\mathcal{K}_j}$ has vanishing trace. In particular, the  Ricci curvature of $R|_{\mathcal{K}_j}$ vanishes.
\end{corollary}
\begin{proof}
Since $\tau|_{\mathcal{K}_j}$ is a holomorphic parallel $(2, 0)$-form, the restriction of the holonomy action on $\mathcal{K}_j$ satisfies $h\in SU(2k_j)$ with $2k_j=\dim_{\mathbb{C}}(\mathcal{K}_j)$. On the other hand, $R_{x, y}|_{\mathcal{K}_j}$ can be obtained as the derivative of one parameter family $h(t)$ (cf. Lemma 2.2 in \cite{Ni-21} and its proof), thus with zero trace. The final statement is due to that  $\Ric_{R|_{\mathcal{K}_j}}(v, \bar{v})=\sum_\ell R(v, \bar{v}, E_{\ell}^j, \bar{E}_{\ell}^j)$ for $v\in \Gamma(\mathcal{K}_j)$ where $\{E_\ell^j\}$ is a unitary base of $\mathcal{K}_j$.
\end{proof}

\vspace{0.3cm}

\section{Symmetric holonomy systems}

Here we work with real numbers and vector bundles over real numbers. To apply the discussion here to the previous section one can consider the realification of complex bundles. Note that the real part of a Hermitian metric on a complex bundle is a Riemannian metric on the realification of the complex bundle.

The concept of holonomy systems was introduced by J. Simons \cite{Sim} in his intrinsic proof of Berger's holonomy theorem for the Riemannian holonomy groups with respect to the Levi-Civita connection. The Riemannian holonomy system is a triple  $S=\{V, \Rm, G\}$, which   consists of,   a Euclidean space $V$ of dimension $\ell$ (we call it the degree of $S$)  endowed with an inner product $\langle \cdot, \cdot\rangle$ (also denoted by a bilinear form $H$),  a connected compact subgroup $G$ of $\mathsf{SO}(\ell)$, and  an algebraic curvature operator $\Rm$ (defined on $V$)  satisfying  the 1st Bianchi identity and that $\Rm_{x, y}\in \mathfrak{g}$, $\forall x, y\in V$ with $\mathfrak{g}\subset \mathfrak{so}(\ell)$ being the Lie algebra of $G$. To see the relevance we first recall the Ambrose-Singer's holonomy theorem \cite{KN} (see also \cite{Ni-21} for a simple proof of half of the result).

\begin{theorem}[Ambrose-Singer]\label{thm:AS}  For any connection $\nabla$ on a Riemannian vector bundle $(E, h)$, let $R$ be its curvature. Given any path $\gamma$ from $q$ to $p$, let $\gamma$ also denote the parallel transport along it. Then $\{\gamma(\Rm^q)\}$ generates the holonomy algebra, where $\gamma(\Rm^q)\in \mathfrak{so}(E_p)$ is defined $\gamma \cdot \left(\Rm_{x, y}|_{E_q}\right) \cdot \gamma^{-1}$. Here $x, y\in T_qM$.
\end{theorem}

In the case that $(M, g)$ is a Hermitian manifold with a CAS structure, with respect to the Chern connection, which preserves the inner product of the underlying real tangent bundle, for any parallel invariant subbundle $\mathcal{K}\subset T'M$, one can apply the above to the bundle $\mathcal{V}=\mathcal{K}\oplus \overline{\mathcal{K}}$. Since $\nabla \Rm=0$, it is easy to see that $\gamma(\Rm^q)=\Rm|_{\mathcal{V}_p}$. Consequently we have

\begin{proposition} \label{prop:31} Let $(M, g)$ be a  Hermitian manifold whose Chern connection is Ambrose-Singer. Let $\mathcal{V}$ be as above. Let $G$ be the restricted holonomy group.
Then $S=(\mathcal{V}_p, \Rm, G)$ is a holonomy system. In fact it is  symmetric.
\end{proposition}
\begin{proof} The first part is a direct consequence of Theorem \ref{thm:AS} and the observation that the curvature $\Rm$ satisfying the first Bianchi identity. The second statement follows from the fact that $\nabla \Rm=0$. Recall that a holonomy system is called symmetric if $\gamma(\Rm)=\Rm$ for any $\gamma\in G$.
\end{proof}

The following result holds the key of the algebraic aspect of the proof.

\begin{proposition}\label{prop:32} Assume that $S=\{V, \Rm, G\}$ is an irreducible symmetric holonomy system.
Then the Ricci flatness of $\Rm$ implies that $\Rm$ is flat.
\end{proposition}

\begin{proof} Here we follow the argument in Theorem 3.1 of \cite{Ni-21}. First we set up the notations and conventions. Identify $\mathfrak{so}(n)$ with $\wedge^2 V$. Here $S^2(\wedge^2 V)$ denotes the symmetric transformations of $\wedge^2V$. The $S^2_B(\cdot)$ denotes the subspace satisfying the 1st Bianchi identity.

Define the metric on $\mathfrak{gl}(V)$ by
$$
\langle A, B\rangle \doteqdot\frac{1}{2}\sum_i\langle A(e_i), B(e_i)\rangle=\frac{1}{2}\operatorname{trace}(B^{tr} A)
$$
In particular for $A, B\in \mathfrak{g}$ the above inner product applies.
Let $P$ be the projection from $\wedge^2(V)$ onto $\mathfrak{g}$, and let $\mathsf{T}:\mathfrak{g}\to \mathfrak{g}$ be the  symmetric isomorphism corresponding to  the negative definite bilinear form on $\mathfrak{g}$
\begin{equation}\label{killing1}
B(A, A')\doteqdot K(A, A')-2\langle A, A'\rangle
\end{equation}
with $K$ being the Killing form of $\mathfrak{g}$ (defined as $K(A, A')=\operatorname{trace}(\ad_A\cdot \ad_{A'})$). Namely $\mathsf{T}$ is defined by $B(A, A')=\langle \mathsf{T}(A), A'\rangle$.   Let $\mathfrak{J}=\mathfrak{g}\oplus V$ (orthogonal sum with the inner product of $V$ and  $\langle A, B\rangle$ on $\mathfrak{g}$ as elements in $\mathfrak{so}(n)$) and define a Lie algebra structure on $\mathfrak{J}$ by letting
$$
[A, A']\doteqdot [A, A']; \quad [x, y]\doteqdot- \Rm_{x, y}; \quad [A, x]\doteqdot A(x),  \forall A, A' \in \mathfrak{g},\,  x, y\in V.
$$
Since $A(\Rm)=0, \forall A\in \mathfrak{g}$, together with the first Bianchi identity, it is easy to check that the bracket so defined satisfies the Jacobi identity, namely $\mathfrak{J}$ is a Lie algebra. Let $B'$ be the Killing form of $J$. It is a  basic result of Lie algebra that $B'$ is $\ad_{J}$-invariant.

{\it Claim 1:  $B'|_{\mathfrak{g}}$  is given by $B$ defined by (\ref{killing1}), hence is  negative definite.}
 By the definition $B'(A, B)=\operatorname{trace} (\ad_A \cdot \ad_B)=\sum_{i=1}^n \langle \ad_A \cdot \ad_B (e_i), e_i\rangle +\sum_{\alpha} \langle \ad_A \cdot \ad_B (A_\alpha), A_\alpha\rangle$ where $\{e_i\}$ ($\{A_\alpha\}$) is an orthonormal frame of $V$ ($\mathfrak{g}$ respectively). The second summand is $K(A, B)$. By the definition of the Lie bracket the first summand is $-\langle B(e_i), A(e_i)\rangle =-2\langle A, B\rangle$. We first need the following computational results.

{\it Claim 2: $B'(A, x)=0$ for $A\in \mathfrak{g}$ and $x\in V$.} Similarly $$B'(A, x)=\operatorname{trace} (\ad_A \cdot \ad_x)=\sum_{i=1}^n \langle \ad_A \cdot \ad_x (e_i), e_i\rangle +\sum_{\alpha} \langle \ad_A \cdot \ad_x (A_\alpha), A_\alpha\rangle$$ where $\{e_i\}$ ($\{A_\alpha\}$) is an orthonormal frame of $V$ ($\mathfrak{g}$ respectively). The  first term vanishes since  $$\langle \ad_A \cdot \ad_x (e_i), e_i\rangle=\langle [A, \Rm_{x, e_i}]_{\mathfrak{g}}, e_i\rangle=0.$$ For the second term, $\langle \ad_A \cdot \ad_x (A_\alpha), A_\alpha\rangle=\langle -A(A_\alpha(x)), A_\alpha\rangle=0$.

{\it Claim 3: $B'|_{V}=\lambda H$, where $H(x, y):=\langle x, y\rangle$, and  $\lambda\ne 0$ if $\Rm\ne 0$.} Since $B'|_V$ is $\ad_\mathfrak{g}$-invariant, hence $G$-invariant.  By  the irreducibility of $G$-action on $V$,  it implies that  $B'(x, y)=\lambda\langle x, y\rangle$ for some $\lambda$.

 If $\lambda=0$, $B'([x, y], [x, y])=B'(x, [y, [x, y]])=0$ since $[y, [x,y]]\in V$.  Now by {\it Claim 1}, which asserts that $B'|_{\mathfrak{g}}$ is negative definite, we have that $[x, y]=-\Rm_{x, y}=0, \forall x, y \in V$. Hence $\Rm=0$.

Below we assume that $\lambda\ne 0$, otherwise we have proved the result by the above argument.
 By Claim 3 we have that
$\langle [[x,y], z], w\rangle=\frac{1}{\lambda} B'([[x,y], z], w)=\frac{1}{\lambda} B'([x,y], [z,w])$, which in turn equals to $\frac{1}{\lambda} B([x,y], [z,w])$ by Claim 1.  Putting them together we have the equation for $x,y,z,w \in V$
\begin{equation} \label{eq:42}
-\langle \Rm_{x, y} z, w\rangle = \langle [[x,y], z], w\rangle = \frac{1}{\lambda} B([x,y], [z,w]\rangle= \frac{1}{\lambda} \langle [x,y], \mathsf{T}([z, w])\rangle.
\end{equation}
Now by (\ref{eq:42}) the Ricci curvature can be expressed as
$$
\Ric_{\Rm}(x, x)=\sum_{i=1}^n \langle \Rm_{x, e_i} e_i, x\rangle =\frac{1}{\lambda}\sum B([x, e_i], [x, e_i]).
$$
Hence if $\Ric_{\Rm}=0$ it implies that $[x, e_i]=-\Rm_{x, e_i}=0$ for any $e_i, x$, namely $\Rm=0$.

Note that a  calculation shows that
$\langle [[x,y], z], w\rangle =-\langle \Rm_{x, y}(z), w\rangle=\langle [x, y], (z\wedge w)\rangle$. Hence by the above (\ref{eq:42})  we have that
$$
\langle [x, y], (z\wedge w)\rangle=\frac{1}{\lambda} \langle [x,y], \mathsf{T}([z, w])\rangle=-\frac{1}{\lambda} \langle [x,y], \mathsf{T}(\Rm_{z,w})\rangle.
$$
This proves that $\Rm_{z,w}=-\lambda \mathsf{T}^{-1} \cdot P(z\wedge w)$, an expression of curvature (due to Kostant) in terms of the Lie algebraic structure.
\end{proof}
The above proposition is  the algebraic counter part  of Theorem 8.6 of Vol II of \cite{KN}.

\vspace{0.3cm}

\section{Hermitian manifolds  with a CAS structure}

Recall that the subbundle $\mathcal{N}_0$, which is defined as
$$
\mathcal{N}_0:= \{X \in \mathcal{N}\ |\   T(X, Y)=0, \ \ \ \forall \ Y\in T^{1,0}M \},
$$
was introduced in \cite{NZ-CAS} to capture any K\"ahler de Rham factor in the universal cover $\widetilde{M}$ of a Hermitian manifold $M$ with a CAS structure. It is easy to see that if there are de Rham K\"ahler factors in $\widetilde{M}$, then the corresponding holomorphic tangent subbundle  must be $\mathcal{N}_0$.
Applying it to the $\widetilde{M}$ we have the following result.

\begin{proposition} Let $(M, g)$ be a Hermitian manifold with a CAS structure. Let $\widetilde{M}$ be its universal cover.
Let $\mathcal{N}_0$ be the subbundle of $T'\widetilde{M}$ defined above. Then $\widetilde{M}$ splits as $M_1\times M_2$ with $M_2$ being K\"ahler and $T'M_2=\mathcal{N}_0$. Moreover $M_1$ does not admit any K\"ahler factor.
\end{proposition}
\begin{proof} This is essentially Theorem 3.6 of \cite{NZ-CAS}.
\end{proof}

By combining  Corollary \ref{coro:21}, Propositions \ref{prop:31}, \ref{prop:32}, we get the proof of our main result, Theorem \ref{thm:Class}:

\begin{proof}[{\bf Proof of Theorem \ref{thm:Class}.}]  First split the manifold $\widetilde{M}$ into two factors as in the statement. By the proof of Lemma 2.2 in \cite{Ni-21}, the splitting of the holonomy action implies that splitting of $\Rm_{x, y}$. Applying Corollary \ref{coro:21}, Propositions \ref{prop:31}, \ref{prop:32} to the irreducible $G$-invariant subbundle of $T'M_1$ we conclude that the Chern curvature is flat on each irreducible summand, hence  is totally flat. Now since $M_1$ has zero curvature and parallel torsion, we may find global parallel vector fields $X_1, \cdots, X_k$ with $k=\dim_{\mathbb{C}}(M_1)$ (which are holomorphic) such that
$$
[X_i, X_j]=T(X_i, X_j)=T_{ij}^k X_k.
$$
The fact that the torsion is parallel implies that $\{T_{ij}^k\}$ are constants. By Lemma 3.1 of \cite{NZ-CAS}, they also satisfy the Jacobi identity in general.   Now we can appeal the result of Cartan \cite{Cartan} (pp. 188-192) to assert that there is a complex Lie group structure on $M_1$. For a modern treatment of the existence of a complex Lie group structure one can see \cite{DFN} pages 20-25, Section 3.1 of Ch1 for the existence of Lie group structure and \cite{Mat}, Theorem on page 212 of Section 11 of Ch IV for details on the existence of complex Lie group structure.
\end{proof}

\begin{proof}[{\bf Proof of Corollary \ref{coro:AK}.}] By Theorem \ref{thm:Class}, the universal cover of $M$ is the product of a complex Lie group $M_1$ with $M_2$ which is products of irreducible simply-connected Hermitian symmetric spaces, namely $M_2=\mathbb{C}^{n_1}\times N_2\times \cdots \times N_\ell$ with $N_i$ being a non-flat irreducible simply-connected Hermitian symmetric space. However, the non-flat factor must be of semi-simple type hence have its Ricci curvature being a nonzero factor of the K\"ahler metric. Then the assumption implies that $\widetilde{M}=M_1\times \mathbb{C}^{n_1}$, which is a complex Lie group itself.
\end{proof}

The corollary shows that $T'M_1$ may miss some complex Lie group factors. To capture it let $\mathcal{F}_p=\{ Z\in T^{1, 0}_pM\, |\, h (Z)=Z, \forall \,h\in G_p\}$, where $G$ is the holonomy group of the Chern connection at $p$. Recall the following result from \cite{NZ-CAS}.
\begin{proposition} Let $(M, g)$ be a Hermitian manifold with a  CAS structure.
Let $\mathcal{F}=\cup_{x} \mathcal{F}_x$ be the sub-bundle of $T^{1, 0}M$. Then $\mathcal{F}$ is a holomorphic integrable foliation. Moreover, if $Z_1, \cdots, Z_r$ is a parallel frame of $\mathcal{F}$, then
$$
[Z_i, Z_j]=c^k_{ij} Z_k
$$
for some constant $c^k_{ij}$. In particular, when $M$ is simply-connected, there exists a complex Lie group $F$ acting almost freely, holomorphically on $M$ such that $T^{1, 0}_x(F\cdot x) =\mathcal{F}_x$.
\end{proposition}

Theorem \ref{thm:Class} implies that on $\widetilde{M}$, $\mathcal{F}=T' M_1\oplus \mathbb{C}^{n_1}$.
The main reason that our argument works is that for a manifold with CAS structure, at the curvature level, algebraically it is the same as a locally symmetric space.

\vspace{0.1cm}

We conclude this section by a discussion on Remark \ref{remark14}, which says that the naive way of generalizing AK type theorem from Levi-Civita conneciton to Chern connection would fail, namely, there are examples of compact locally homogeneous Hermitian manifold whose Chern connection has vanishing (third) Ricci curvature but is not Chern flat. This illustrates that the  CAS assumption in Corollary \ref{coro:AK} is necessary. We will give two such examples below.

The first  example is some special type of {\em almost abelian manifolds,} namely,  compact quotients of $(G,J,g)$, where $G$ is an even-dimensional, connected and simply-connected, unimodular  Lie group, $J$ a left-invariant (integrable) complex structure on $G$ and $g$ a left-invariant metric on $G$ compatible with $J$. Let ${\mathfrak g}$ denote the Lie algebra of $G$. $G$ is said to be {\em almost abelian} if ${\mathfrak g}$ contains an abelian ideal  ${\mathfrak a}$ of codimension $1$. The Hermitian structures on almost abelian Lie algebras were studied by a number of authors, and we refer the readers to \cite{AL, Bock, FinoP21, FinoP23, LW} and the references therein for more information on this. We will follow the computation in \S 3 of \cite{GZ} to serve our purpose here.

A unitary basis $\{ e_1, \ldots , e_n\}$ of ${\mathfrak g}^{1,0}=\{ x-iJx \mid x \in {\mathfrak g}\}$ is called {\em admissible} if ${\mathfrak a}$ is spanned by $e_j+\overline{e}_j$, $i(e_j-\overline{e}_j)$ for $2\leq j\leq n$ and $i(e_1-\overline{e}_1)$. The structural constants $C$ and $D$ are defined by
$$ [e_i, e_j] = \sum_k C^k_{ij} e_k, \ \ \ \ \ \ [e_i, \overline{e}_j] = \sum_k \big( \overline{D^i_{kj}}\, e_k - D^j_{ki} \,\overline{e}_k \big) . $$
When $e$ is admissible, the only possibly non-zero components of $C$ and $D$ are
$$ D^1_{11}=\lambda \in {\mathbb R}, \ \ \ \ D^1_{i1}=v_i \in {\mathbb C}, \ \ \ \ D^j_{i1} =A_{ij}, \ \ \ \ C^j_{1i}=-\overline{A_{ji} }, \  \ \ \ \forall \ 2\leq i,j\leq n. $$
From the calculation in \S 3 of \cite{GZ}, we know that for the almost abelian Lie group $(G,J,g)$,
\begin{itemize}
\item $G$ is unimodular \ \ $\Longleftrightarrow$ \ \ $\lambda + \mbox{tr}(A) + \overline{\mbox{tr}(A)} =0$;
\item  $g$ is Chern flat  \ \ $\Longleftrightarrow$ \ \ $\lambda =0$, $v=0$, and $[A, A^{\ast}]=0$;
\item $g$ has vanishing third Chern Ricci \ \ $\Longleftrightarrow$ \ \ $\lambda =0$, $v=0$, and $\mbox{tr}(A) + \overline{\mbox{tr}(A)} =0$.
\end{itemize}
Here $A^{\ast}$ stands for the conjugate transpose of $A$. With these notations and set-ups, our first example goes like the following:

{\em For any $n\geq 3$, if we choose the $(n-1)\times (n-1)$ matrix $A$ so that $\mbox{tr}(A) + \overline{\mbox{tr}(A)} =0$ but $[A, A^{\ast}]\neq 0$, then the metric $g$ would have vanishing third Chern Ricci but is not Chern flat.} (Note that the first Chern Ricci vanishes here, but the second Ricci does not vanish).

\vspace{0.15cm}

Another (simpler) example is the (compact quotients of) nilpotent Lie groups with a left-invariant Hermitian structure that is balanced and with parallel Bismut torsion. Denote by $G$ the Lie group and ${\mathfrak g}$ its Lie algebra. Under some unitary basis $e$ of ${\mathfrak g}^{1,0}$  the structure equation takes the form:
\begin{equation}
 \left\{ \begin{split}
d\varphi_i\,=\,0, \ \ \ \ \  \forall \ 1\leq i\leq r; \hspace{2.6cm} \\
d\varphi_{\alpha}=\sum_{i=1}^r Y_{\alpha i} \,\varphi_i \wedge \overline{\varphi}_i, \ \ \ \ \forall \ r+1\leq \alpha \leq n, \end{split} \right.
\end{equation}
where $2n$ is the real dimension of $G$, $1\leq r< n$, $\varphi$ is the coframe dual to $e$, and $Y_{\alpha i}$ are arbitrary complex constants such that $\sum_{i=1}^r Y_{\alpha i}=0$ for each $\alpha$.

{\em When these $Y_{\alpha i}$ are not all zero, the Chern connection is not flat. On the other hand, the metric is balanced,  the first and third Chern Ricci both vanish} (but the second Chern Ricci is not identically zero).

\vspace{0.15cm}

To conclude the discussion of this section, let us remark that we do not know any example which will serve as an negative answer to the following question:

\begin{question}
Let $(M,g)$ be a compact locally homogeneous Hermitian manifold. If the first, second, and third Chern Ricci all vanish, then must it be Chern flat?
\end{question}

We do not even know the answer to this question for Lie-Hermitian manifolds, namely when the universal cover of $M$ is a Lie group (not necessarily a complex Lie group) equipped with left-invariant complex structure and left-invariant metric.

\vspace{0.3cm}

\section{Generalized symmetric holonomy systems}
In this section we consider the symmetric holonomy system for an affine connection with a torsion. Suppose now that we have a metric connection $\nabla$ with torsion. For Hermitian manifold we also require that it satisfies $\nabla J=0$, where $J$ is the almost complex structure.
For the situation that $\nabla$ is Ambrose-Singer (or invariant under the parallelism), we have that $g(R)=R$ and $g(T)=T$ for any element $g$ belonging to the restricted holonomy. Not that $T: \wedge^2 T_pM\to T_pM$ for any given $p\in M$.

A generalized symmetric holonomy system is   $S=\{V, \Rm, T, G\}$, which   consists of,   a Euclidean space $V$ of dimension $n$ (we call it the degree of $S$)  endowed with an inner product $\langle \cdot, \cdot\rangle$ (also denoted by a bilinear form $H$),  a connected compact subgroup $G$ of $\mathsf{SO}(\ell)$,   an algebraic curvature operator $\Rm:\wedge^2 V \to \wedge^2 V$ as a linear map, and a linear map $T:\wedge^2 V \to V$, together satisfying
\begin{eqnarray}
&\,& \Rm_{x, y}\in \mathfrak{g}; \label{eq:holo0}\\
&\,& {\mathfrak S} \{ \Rm_{x, y} z+ T(x, T(y, z))\}=0;\label{eq:B1}\\
&\,& {\mathfrak S}\{ \Rm_{x, T(y, z)}\}=0; \label{eq:B2}\\
&\,& g(\Rm)=\Rm, g(T)=T, \forall \, g\in G, \mbox{ or  equivalently } A(T)=A(\Rm)=0, \forall A \in \mathfrak{g}. \label{eq:GS}
\end{eqnarray}
Here $\mathfrak{g}$ is the Lie algebra of $G$, ${\mathfrak S}$ is the cyclic permutation operation on the positions of $x, y, z$, the actions of $A$ on $T$ and $\Rm$ are the ones induced as derivations. Namely, $A(T)(x, y)=AT(x, y)-T(Ax, y)-T(x, Ay)$ and $(A(\Rm))_{x, y}=A \cdot \Rm_{x, y}-\Rm_{Ax, y}-\Rm_{x, Ay}-\Rm_{x, y}\cdot A.$

The following result, which is an immediate consequence of (\ref{eq:B1}), (\ref{eq:B2}) and (\ref{eq:GS}),  is due to Nomizu \cite{Nomizu}.

\begin{theorem}[Nomizu]\label{thm:Nomizu} Let $\mathfrak{J}=V\oplus \mathfrak{g}$. Define a Lie algebra structure on $\mathfrak{J}$ by letting
$$
[A, A']\doteqdot [A, A']; \quad [x, y]\doteqdot- \Rm_{x, y}\oplus (-T(x, y)); \quad [A, x]\doteqdot A(x),  \forall A, A' \in \mathfrak{g},\,  x, y\in V.
$$
Then $\mathfrak{J}$ is a Lie algebra.
\end{theorem}

From the definitions in the theorem one can easily compute the Lie bracket of $[x\oplus A, y\oplus A']$.
Let $B'$  still denote the Killing form of the Lie algebra $\mathfrak{J}$, which is a bilinear form. We also denote the bilinear form $-2\langle\cdot, \cdot\rangle+K(\cdot, \cdot)$ on $\mathfrak{g}$ by $B(\cdot, \cdot)$, which is negative definite. Here $K$ is the Killing form of $\mathfrak{g}$ which is nonnegative due to the compactness of $G$. The results in Section 3 mostly no longer hold for the generalized symmetric holonomy system. Recall that we endow $\mathfrak{J}$ with metrics from $\mathfrak{g}\subset \mathfrak{so}(n)$ and $V$. However we still have the following result.

\begin{lemma}\label{lm:killing1} The $B'|_{\mathfrak{g}}$ is the same as $B$, namely
$B'|_{\mathfrak{g}}=B$. In particular $B'|_{\mathfrak{g}}$ is negative definite. If we assume additionally that $V$ is irreducible (namely as a $G$-module), then $B'|_{V}(\cdot, \cdot)=\lambda H(\cdot, \cdot)$.
\end{lemma}
\begin{proof} Direct calculation shows,  for an orthonormal frame $\{e_i\}_{i=1}^{n}$ of $V$  and an orthonormal frame  $\{A_\alpha\}_{\alpha=1}^{\dim(\mathfrak{g})}$  of $\mathfrak{g}$, that
\begin{eqnarray*}
B'(C_1, C_2)&=&\operatorname{trace} (\ad_{C_1} \cdot \ad_{C_2})\\
&=&\sum_{i=1}^n \langle \ad_{C_1} \cdot \ad_{C_2} (e_i), e_i\rangle +\sum_{\alpha} \langle \ad_{C_1} \cdot \ad_{C_2} (A_\alpha), A_\alpha\rangle\\
&=&-\sum_{i=1}^n \langle C_2(e_i),  C_1(e_i)\rangle+K(C_1, C_2)\\
&=&B(C_1, C_2).
\end{eqnarray*}
Here $C_1, C_2\in \mathfrak{g}$ are two arbitrary elements. This proves the claimed identity. The last claimed follows from Schur's lemma.
\end{proof}
For $C\in \mathfrak{g}$, $x\in V$,  $B'(C, x)$  only vanishes under additional conditions.

\begin{lemma}\label{lm:killing2}
$B'(C, x)=2\langle T(x, \cdot), C\rangle.$ Here we view $T(x, \cdot)$ as an endomorphism of $V$. In particular, for a Ambrose-Singer connection  it vanishes if $\Rm_{x,y}\cdot T=0$.
\end{lemma}
\begin{proof} By Ambrose-Singer theorem and its specialization, Lemma 4 of \cite{Kostant-Nagoya} we have that the holonomy algebra $\mathfrak{g}$ is generated by $\sum_j \Rm_{x_j, y_j}$ for some $x_j, y_j\in V$. Assume that we have
$B'(C, x)=2\langle T(x, \cdot), C\rangle$. Writing $C=\sum_j \Rm_{x_j, y_j}$ we have that
\begin{eqnarray*}
B'(C, x)&=&\sum_{i=1}^n \sum_j \langle T(x, e_i), \Rm_{x_j, y_j} e_i\rangle\\
&=&-\sum_{i=1}^n \sum_j \langle\Rm_{x_j, y_j}T(x, e_i), e_i\rangle \\
&=&0.
\end{eqnarray*}
This proves the second statement assuming the first. Now we prove the first part, namely the claimed identity in the lemma. Direct calculation shows that
\begin{eqnarray*}
B'(C, x)&=& \sum_i \langle \ad_C \ad_{x}(e_i), e_i\rangle +\sum_{\alpha} \langle \ad_C \ad_x (A_\alpha), A_\alpha\rangle\\
&=& -\sum_i \langle C(T(x, e_i)), e_i\rangle -\langle [C, \Rm_{x,e_i}], e_i\rangle -\sum_{\alpha} \langle C(A_\alpha(x)), A_\alpha\rangle \\
&=& \sum_{i}\langle T(x, e_i), C(e_i)\rangle.
\end{eqnarray*}
This proves the  claimed equation.
\end{proof}

We say $\Rm|_{\operatorname{image} (T)}=0$ if $\Rm_{x, y} \cdot T=0$ and $R_{x, T(y, z)}=0$. Now we  prove the following generalization of Proposition \ref{prop:32}.

\begin{theorem}\label{thm:generlizedholo} Let $S=\{V, \Rm, T, G\}$ be a generalized symmetric  holonomy system. Assume that it is irreducible. Assume further  that $\Rm|_{\operatorname{image} (T)}=0$, and either 

(i)  $T(x, \cdot)$ is skew-symmetric, or 

(ii) (a) $V$ is Hermitian with an almost complex structure $J$ (hence $V=V^{1,0}\oplus V^{0,1}$); (b) $\Rm( \wedge^{1,1} V) \subset \wedge^{1,1} V$ and vanishes on the other components; (c) $T$ splits as  $T: V^{1,0}\times V^{1,0}\to V^{1,0}$ and $T: V^{0,1} \times V^{0,1} \to V^{0, 1}$ with all other components being zero. 

Then $\Ric_{\Rm}(x,x)=0$  implies that $\Rm=0$.
\end{theorem}
\begin{proof} Note that $B'|_{V}(\cdot, \cdot)=\lambda H(\cdot, \cdot)$. If $\lambda=0$, then $B'|_{V}\equiv 0$.
Observe, by the assumption $\Rm|_{\operatorname{image} (T)}=0$, that
\begin{eqnarray*}
[y, [x, y]]&=&[y, -\Rm_{x, y}\oplus (-T(x, y))]\\
&=&\Rm_{y, T(x, y)}\oplus (\Rm_{x,y} y+T(y, T(x,y)))\\
&=& \Rm_{x,y} y+T(y, T(x,y))\in V.
\end{eqnarray*}
Hence from $B'([x,y], [x,y])=B'(x, [y, [x,y]])=B'(x, \Rm_{x, y}y+T(y, T(x,y))=0$ we have that
$$B'([x,y], [x,y])=0.$$
On the other hand by Lemma \ref{lm:killing2} we also have that
\begin{eqnarray*}
B'([x,y], [x,y])&=& B'(\Rm_{x,y}\oplus T(x,y), \Rm_{x,y}\oplus T(x,y))\\
&=& B'(\Rm_{x, y}, \Rm_{x, y})<0
\end{eqnarray*}
unless $\Rm_{x, y}=0$. Here we have used that Lemma \ref{lm:killing2} implies $B'((\Rm_{x,y}, T(x, y))=0$. This proves the theorem for the case $\lambda=0$.

Now we assume that $\lambda\ne 0$. The first part argument above then shows that
\begin{equation}
\langle \Rm_{x, y} y, x\rangle +\langle T(y, T(x,y)), x\rangle= \frac{1}{\lambda} B'(x, [y, [x,y]])=\frac{1}{\lambda}B'([x,y], [x,y]).
\end{equation}
Now using $\Ric_{\Rm}(x,x)=\sum \langle \Rm_{ \epsilon_i, x}x, \epsilon_i\rangle$ (with $\{\epsilon_i\}$ being an orthogonal basis of $V$) we have that, in the case (i),
\begin{eqnarray*}
\langle T(x, \epsilon_i), T(x, \epsilon_i)\rangle+\Ric_{\Rm}(x,x)&=&\langle T(\epsilon_i, T(x,\epsilon_i)), x\rangle +\Ric_{\Rm}(x,x)\\
&=&\frac{1}{\lambda}\sum_{i=1}^n B'([x,\epsilon_i], [x,\epsilon_i])\\
&=&\frac{1}{\lambda} \sum_{i=1}^n B'(\Rm_{x,\epsilon_i}\oplus T(x,\epsilon_i), \Rm_{x,\epsilon_i}\oplus T(x,\epsilon_i))\\
&=& \frac{1}{\lambda}\sum_{i=1}^n  B(\Rm_{x,\epsilon_i}, \Rm_{x,\epsilon_i})+\sum_{i=1}^n \langle T(x, \epsilon_i), T(x, \epsilon_i)\rangle.
\end{eqnarray*}
Canceling the same term on the both sides we have   $\Ric_{\Rm}(x,x)=\frac{1}{\lambda}\sum_{i=1}^n  B(\Rm_{x,\epsilon_i}, \Rm_{x,\epsilon_i})$, which implies that  $\frac{1}{\lambda}\sum_{i=1}^n  B'(\Rm_{x,\epsilon_i}, \Rm_{x,\epsilon_i})=0$ if $\Ric_{\Rm}=0$. The result then follows from Lemma \ref{lm:killing1}, namely that $B'|_{\mathfrak{g}}$ is negatively definite. This proves the result for case (i). Case (ii) is similar. This completes the proof of the theorem.
\end{proof}

\vspace{0.05cm}

\section{Bismut Ambrose-Singer manifolds}

From the discussion in the previous sections, we obtained a complete classification of compact Hermitian manifolds that is CAS, namely, whose Chern connection is Ambrose-Singer. Furthermore, we also know that Chern connection satisfies the Alekseevski\u{i}-Kimel\!\'{}\!fel\!\'{}\!d type theorem, namely, a CAS metric is flat if it is Ricci flat. It is natural to ask what happens if we replace the Chern connection $\nabla$ by the Bismut connection $\nabla^b$, or more generally by the $t$-Gauduchon connection $\nabla^{(t)}$, which is the linear combination $(1-\frac{t}{2}) \nabla + \frac{t}{2} \nabla^b$ where $t\in {\mathbb R}$.

Recall that in \S 1 we called a Hermitian metric {\em BAS} if its $\nabla^b$ is Ambrose-Singer, namely, both the torsion and curvature of $\nabla^b$ are parallel with respect to $\nabla^b$ itself. Similarly, we called a Hermitian metric {\em $t$-GAS} if its $t$-Gauduchon connection $\nabla^{(t)}$ has parallel torsion and curvature with respect to $\nabla^{(t)}$ itself. We would like to understand the set of all BAS (or more generally, all $t$-GAS) manifolds, and know in particular whether or not the Alekseevski\u{i}-Kimel\!\'{}\!fel\!\'{}\!d type theorem will hold for $\nabla^b$ (or  $\nabla^{(t)}$).

We first make some remarks on a complete simply connected BAS manifold. By $\nabla^b T^b=0$ and that $T^b$ is either the same or differs by a sign from $T$ we can have that $\nabla^{b}T=0$. Recall the tensors $S_XY:=\nabla_XY-\nabla^{LC}_X Y$ and $S^b_X Y :=\nabla^b_XY-\nabla^{LC}_X Y$ and that they can be expressed in terms of the torsions we have that $\nabla^b S^b =\nabla^b S=0$. Also observe that the metric connections are complete, we can evoke Theorem 4 of \cite{Kostant-Nagoya} to obtain a  simply-connected Lie group $G$ acting as $\nabla^b$, $\nabla$, and $\nabla^{LC}$ affine transformations on $M$ transitively, whose Lie algebra is given by $\mathfrak{J}$ defined in the last section in terms of the holonomy algebra of the Bismut connection.

For the Bismut connection $\nabla^b$ we  denote the first, second and third Ricci curvatures as $\Ric^{b(i)}$, $i=1, 2,3$. In \S 1, we raised the following (= Conjecture \ref{conj1.6a}):

\begin{conjecture} \label{conj6.1}
Suppose $(M^n,g)$ is a compact Hermitian manifold that is BAS. Assume that the first and third Bismut Ricci both vanish. Then it must be Bismut flat.
\end{conjecture}

In complex dimension $2$, the answer to Conjecture \ref{conj6.1} is positive. As a direct consequence of \cite[Theorem 2]{ZZpluriclosed}, we have  the following
\begin{proposition}
A compact, non-K\"ahler Hermitian surface $(M^2,g)$ is BAS if and only if it is Vaisman with constant scalar curvature.
\end{proposition}

Recall that a Hermitian manifold is said to be {\em Vaisman} if it is locally conformally K\"ahler and its Lee form is parallel under the Levi-Civita connection. Compact non-K\"ahler Vaisman surfaces are fully classified by Belgun in a beautiful work \cite{Belgun}. They are either non-K\"ahler properly elliptic surfaces, or Kodaira surfaces, or  Class 1 or elliptic Hopf surfaces.

We remark that for a non-K\"ahler Vaisman surface, the scalar curvature of Levi-Civita or Chern or Bismut connection (or the trace of first or second Ricci of Chern or Bismut connection) all differ by constants. So  the word `constant scalar curvature' here simply means any one of them (hence all of them) is constant. Also, such a surface always admits local unitary frame $\{ e_1, e_2\}$ under which the only possibly non-zero component of the Bismut curvature tensor is $R^b_{1\bar{1}1\bar{1}}$. So clearly if it is Bismut Ricci flat (or equivalently Bismut scalar flat) then it will be Bismut flat:

\begin{corollary}
If a compact BAS surface is Bismut Ricci flat, then it is Bismut flat.
\end{corollary}

Here we could drop the adjective `non-K\"ahler' since any locally Hermitian symmetric space would be flat if Ricci flat. Note that the only compact, non-K\"ahler Bismut flat surfaces are the isosceles Hopf surfaces (\cite{WYZ}). The above statement confirms Conjecture \ref{conj6.1} in the $2$-dimensional case.

\vspace{0.1cm}

Now let us move on to higher dimensions. Recall that {\em Bismut torsion-parallel} (or {\em BTP} for short) manifolds are Hermitian manifolds whose Bismut connection has parallel torsion. Such manifolds form a rather interesting class, and contains both our BAS manifolds and the BKL (Bismut K\"ahler-like) manifolds. Compact, non-K\"ahler BTP manifolds in dimension $3$ were characterized in \cite{ZZ22}, from which we could pick out those with vanishing first and third Bismut Ricci, and they all turn out to be Bismut flat, proving Conjecture \ref{conj6.1} in the $n=3$ case.

Recall that the $B$-tensor of a Hermitian manifold $(M^n,g)$ is defined by $B_{i\bar{j}}=\sum_{r,s=1}^n T^j_{rs} \overline{ T^i_{rs} }$ under any unitary frame $e$, where $T_{ik}^j$ are components of the Chern torsion. Then
$  \sigma_B= \sqrt{-1} \sum_{i,j=1}^n    B_{i\bar{j}}   \varphi_i \wedge \overline{\varphi}_j$
is a globally defined non-negative $(1,1)$-form on $M$. Here $\varphi$ is the coframe dual to $e$. Also recall that the (first) Chern Ricci form is the global $(1,1)$-form given by
$$ \rho = \sqrt{-1} \sum_{i,j=1}^n \mbox{Ric}^{(1)}_{i\bar{j}}  \varphi_i \wedge \overline{\varphi}_j = - \sqrt{-1} \partial \overline{\partial} \,\log \omega^n ,$$
where $\omega$ is the K\"ahler form of $g$. The following technical lemma is useful in picking out Bismut Ricci flat BAS metrics:

\begin{lemma}  \label{lemma6.4}
Suppose $(M^n,g)$ is a BTP manifold that is (third) Bismut Ricci flat. Then the (first) Chern Ricci form is equal to $\sigma_B$.
\end{lemma}

\begin{proof}
Under any local unitary frame $e$, let us denote by $T_{ik}^j$, $R_{i\bar{j}k\bar{\ell}}$, and $R^b_{i\bar{j}k\bar{\ell}}$ the components of the Chern torsion, Chern curvature, and Bismut curvature. Here we follow the convention that the Chern torsion tensor is $T(e_i,e_k)=\sum_j T^j_{ik} e_j$,  so our $T_{ik}^j$ equals to twice of that in \cite{ZZ22}. Under the BTP assumption, we have $\nabla^bT=0$, so by formula (3.2) and (3.4) in Lemma 3.1 of \cite{ZZ22}, we get the following
\begin{eqnarray}
R^b_{i\bar{j}k\bar{\ell}}- R^b_{k\bar{j}i\bar{\ell}} & = & - T^r_{ik}\overline{ T^r_{j\ell } } - T^j_{ir}\overline{ T^k_{\ell r} } - T^{\ell}_{kr}\overline{ T^i_{jr} } + T^{\ell}_{ir}\overline{ T^k_{jr} } + T^j_{kr}\overline{ T^i_{\ell r} },  \label{eq:6.4} \\
R^b_{i\bar{j}k\bar{\ell}}- R_{i\bar{j}k\bar{\ell}} & = & - T^r_{ik}\overline{ T^r_{j\ell } } - T^j_{ir}\overline{ T^k_{\ell r} } - T^{\ell}_{kr}\overline{ T^i_{jr} } + T^{\ell}_{ir}\overline{ T^k_{jr}}     \label{eq:6.5}
\end{eqnarray}
for any $1\leq i,j,k,\ell \leq n$. Here $r$ is summed from $1$ to $n$ on the right hand side of the above two equations. Taking the difference, we get $R_{i\bar{j}k\bar{\ell}} - R^b_{k\bar{j}i\bar{\ell}} = T^j_{kr}\overline{ T^i_{\ell r} }$. Letting $k=\ell$ and sum it up, we get $\mbox{Ric}^{(1)} - \mbox{Ric}^{b(3)} = B$. So $\rho = \sigma_B$ when $\mbox{Ric}^{b(3)}=0$. This proves the lemma.
\end{proof}

In particular, we have $d\sigma_B=0$ since $d\rho =0$, which gives us the following:

\begin{lemma}
Let $(M^n,g)$ be a (third) Bismut Ricci flat BTP manifold. Then under any local unitary frame it holds that
\begin{equation} \label{eq:BT}
 \sum_r \big( T^j_{rk} B_{i\bar{r}} +  T^j_{ir} B_{k\bar{r}} -  T^r_{ik} B_{r\bar{j}} \big)  \, = \, 0, \ \ \ \ \ \forall \ \,1\leq i,j,k,  \leq n.
 \end{equation}
 In particular, $M$ has non-negative first Chern class but is not Fano.
\end{lemma}

\begin{proof}
Equation (\ref{eq:BT}) is a direct consequence of the structure equations (see also (3.9) of \cite{NZ-CAS}) and the fact that $d\sigma_B=0$, so we just need to prove the `not Fano' part.  Denote by $b_1\geq \cdots \geq b_n\geq 0$ the eigenvalues of $B$. Under the BTP assumption, we have $\nabla^bB=0$, so each $b_i$  is a global constant. If we choose $e$ so that $B$ is diagonal, then (\ref{eq:BT}) becomes
\begin{equation} \label{eq:BT2}
 \big( b_i+b_k-b_j \big) T^j_{ik}  = 0, \ \ \ \ \ \forall \ \,1\leq i,j,k,  \leq n.
 \end{equation}
We claim that $b_n$ must be zero. Assume otherwise, by letting $j=n$ in (\ref{eq:BT2}), we get $T^n_{ik}=0$ for any $i,k$, which leads to $b_n=B_{n\bar{n}} = \sum_{r,s} |T^n_{rs}|^2 = 0$, a contradiction. So $B\not> 0$ thus $M$ is not Fano.
\end{proof}

Now let us prove  the following statement (= Theorem \ref{prop1.7a}):

\begin{theorem} \label{prop6.11}
Any complete, non-K\"ahler BAS manifold with vanishing first and third Bismut Ricci must be non-balanced. Moreover, if $n\le 4$ then it  must be Bismut flat.
\end{theorem}

\begin{proof}
Let $(M^n,g)$ be a complete BAS manifold with vanishing first and third Bismut Ricci. With the unitary frame and its dual as before, consider the $\phi$-tensor which is defined by $\phi_i^j=\sum_r T^j_{ir} \overline{\eta}_r$ under any unitary frame, where $\eta$ is the Gauduchon torsion $1$-form which is defined as $\eta=\sum_{i=1}^n \eta_i \varphi_i$ and $\eta_i=\sum_{k=1}^n T^k_{ki}$. By formula (\ref{eq:6.4}), in the proof of Lemma \ref{lemma6.4}, we see  that vanishing of first and third Bismut Ricci imply that  and $B=\phi + \phi^{\ast}$. On the other hand, Lemma \ref{lemma6.4} implies that $\Ric^{(1)}=B$. Since the trace of the $B$ tensor is equal to the square norm of the Chern torsion, which does not vanish identically as $g$ is not K\"ahler, so we have that $\phi \not\equiv 0$,  hence $\eta \not\equiv 0$, namely, $g$ is not balanced.

By \cite[Proposition 5.2]{ZZ22}, around  any given point in $M$ there always exists an {\em admissible} frame, which is a local unitary frame $e$ under which
$$ \eta_1=\cdots =\eta_{n-1}=0, \ \ \ \eta_n=\lambda >0, \ \ \ T^j_{in} = \delta_{ij}a_i, \ \ \ \forall \ 1\leq i,j\leq n,  $$
where $a_n=0$ and $a_1 +\cdots + a_{n-1}=\lambda$. Since $\eta$, $B$, $\phi$ are all parallel under $\nabla^b$, the number $\lambda = |\eta|>0$ and each $a_i$ are global constants on $M$, as these $\lambda a_i$ are eigenvalues of $\phi$. Denote by $\theta^b$, $\Theta^b$ the matrix of connection and curvature of the Bismut connection $\nabla^b$ under an admissible frame $e$. Since $e_n$ is lined up with $\eta$, we know that $\nabla^be_n=0$ thus $\theta^b_{n\ast}=0$ and $\Theta^b_{n\ast}=0$. By $\nabla^bT=0$, we get
 \begin{equation} \label{eq:6.8}
 0 = d(\delta_{ij}a_i) = dT^j_{in} = \sum_r \big( T^j_{rn} \theta^b_{ir} + T^j_{ir} \theta^b_{nr} - T^r_{in} \theta^b_{rj} \big) = (a_j-a_i)\,\theta^b_{ij} .
 \end{equation}
 Write $b_i=B_{i\bar{i}}= \sum_{r,s} |T^i_{rs}|^2$, we have
\begin{equation} \label{eq:6.9}
b_i = \lambda (a_i+\overline{a}_i) \, = \, 2|a_i|^2 + 2\!\!\sum_{1\leq j<k<n} |T^i_{jk}|^2 \, = \,2|a_i|^2 + 2 \delta_i.
\end{equation}
The second part of the theorem  follows from the following claim and establishment of the assumption in the claim afterwards.

{\bf Claim:} {\em If there exists an admissible frame under which $\theta^b$ is diagonal, and $n\leq 4$, then $R^b=0$.}

To prove the claim, note that by \cite[Theorem 1.1]{ZZ22} any BTP manifold always satisfies the symmetry condition $R^b_{i\bar{j}k\bar{\ell}}=  R^b_{k\bar{\ell}i\bar{j}}$, so the only possibly non-trivial Bismut curvature components are given by
$$ \Theta^b_{ii} = \sum_{k<n} S_{ik} \,\varphi_k \wedge \overline{\varphi}_k, \ \ \ \ \ \forall \ i<n, $$
where $S$ is a symmetric real $(n-1)\times (n-1)$ matrix. The vanishing of the first and third Bismut Ricci implies that the diagonal entries of $S$ vanish, also the sum of each row vanishes. When $n\leq 4$, this will force $S=0$, hence $R^b=0$. This establishes the {\bf Claim}.

Now we prove the assumption of the {\bf Claim} holds when $n\le 4$. Let us assume $n=4$ (the $n=3$ case is analogous but much easier). When $\{ a_1, a_2, a_3\}$ are all distinct, by (\ref{eq:6.8}) we know that $\theta^b$ is diagonal, so by the {\bf Claim} we get $R^b=0$. If $a_1=a_2=a_3$, then $a_1=\frac{\lambda}{3}$ as their sum needs to be $\lambda$, hence $b_1=b_2=b_3=\frac{2}{3}\lambda^2$. On the other hand, by (\ref{eq:BT2}) we know that $T^j_{ik}=0$ for any $1\leq i,j,k\leq 3$, thus (\ref{eq:6.9}) would yield  $\frac{2}{3}\lambda^2= \frac{2}{9}\lambda^2 +0$, which is a contradiction. We are therefore left with the case when $a_1=a_2\neq a_3$. By (\ref{eq:6.8}), we know that $\theta^b$ is block-diagonal, with a $2\times 2$ block on the upper left corner. Let us write $a_1=a$ and $b_1=b$, then  by (\ref{eq:6.9}) we get
\begin{eqnarray*}
 b & =  & 2\lambda \mbox{Re}(a) \ \ = \ \  2|a|^2 + 2 \delta;      \\
 b_3 & =  & 2\lambda \mbox{Re}(a_3) \ \ = \ \  2|a_3|^2 + 2 \delta_3;    \\
\delta & = & |T^1_{12}|^2 + |T^1_{13}|^2 + |T^1_{23}|^2 \ \ = \ \ |T^2_{12}|^2 + |T^2_{13}|^2 + |T^2_{23}|^2;  \\
\delta_3 & = & |T^3_{12}|^2 + |T^3_{13}|^2 + |T^3_{23}|^2.
\end{eqnarray*}
Since $a_3=\lambda -2a$ and $b_3=2\lambda^2 -2b$, the first two equations above give
\begin{equation} \label{eq:6.10}
b = 2\lambda \mbox{Re}(a) =  2|a|^2 + 2 \delta = 4|a|^2 + \delta_3.
\end{equation}
If $a=0$, then $b=0$ and $\delta =\delta_3=0$, so $T^i_{\ast \ast} = T^{\ast}_{i\ast}=0$ for $i=1,2$. This means that $E=\mbox{span}\{ e_1, e_2\}$ is contained in the kernel of $T$, thus is parallel under the Levi-Civita connection. So $E$ gives us a K\"ahler de Rham factor, contradicting with our assumption at the beginning of the proof. Therefore we may assume that $a\neq 0$, or equivalently, $b>0$.

If $b<\lambda^2$,  or equivalently, if $b_3>0$, then by (\ref{eq:BT2}) we know that the only possibly non-zero $T^j_{ik}$ for $1\leq i,j,k\leq 3$ is $T^3_{12}$. In particular, $\delta=0$ and $\delta_3=|T^3_{12}|^2$. By (\ref{eq:6.10}), $\delta_3= - 2|a|^2 <0$ which is impossible.

Now assume that $b=\lambda^2$. The three eigenvalues of the $B$ tensors (we always ignore the direction $e_n=e_4$ here) are $\{ \lambda^2, \lambda^2, 0\}$. In this case $b_3=0$, hence $a_3=0$ and $a=\frac{\lambda }{2}$, $\delta_3=0$, $\delta =\frac{\lambda^2}{4}$. By (\ref{eq:BT2}) we know that the only possibly non-zero $T^j_{ik}$ for $1\leq i,j,k\leq 3$ are $T^j_{i3}=P_{i\bar{j}}$, where $1\leq i,j\leq 2$. The trace of the $2\times 2$ matrix $P$ is zero, as $T^1_{13}+T^2_{23}=\eta_3=0$. Also, by (\ref{eq:6.10}) we know that  the sum of norm square of each column of $P$ is equal to $\delta = \frac{\lambda^2}{4}$. In particular, the two off diagonal elements of $P$ have equal norm. Let us write
$$ P = \left[ \begin{array}{cc} x & y \\ \rho y & -x \end{array} \right], $$
where $|\rho |=1$ and $|x|^2+|y|^2 = \frac{\lambda^2}{4}$. Note that for any unitary change of $\{ e_1, e_2\}$, the new frame $\{ \tilde{e}_1, \tilde{e}_2, e_3, e_4\}$ is still admissible, and $P$ is changed to $\tilde{P}=UPU^{\ast}$ where $U$ is the $\mathsf{U}(2)$-valued function changing $\{ e_1, e_2\}$ to $\{ \tilde{e}_1, \tilde{e}_2\}$. If $y\neq 0$, then in a small neighborhood we may chose $U$ so that the upper right entry $\tilde{y}$ of $\tilde{P}$ is zero. The lower left entry of $\tilde{P}$ is then zero as $\tilde{P}$ is again tracefree and with equal sum of norm square for each column. In other words, by performing a unitary change of $\{e_1, e_2\}$ if necessary, then under the new admissble frame we may assume that $P$ is diagonal, with $|T^1_{13}|=|x|=\frac{\lambda}{2}$ being a global constant. Replace $e_3$ by $\rho_3e_3$ for a suitable function with $|\rho_3|=1$, we may further assume that $T^1_{13}=|T^1_{13}|$, hence is a global positive constant. In summary, we may choose admissble frame so that the only non-zero torsion components are
$$ T^1_{14}=T^2_{24}=a, \ \ \ \ T^1_{13}=-T^2_{23} = a, \ \ \ \ \ a=\frac{\lambda}{2}. $$
We already know that $\theta^b$ is block diagonal, with blocks of size $2$, $1$, $1$, and $\theta^b_{44}=0$. Since $\nabla^bT=0$, we have
$$ 0 =  dT^1_{13} = \sum_r \big( T^1_{r3}\theta^b_{1r} + T^1_{1r}\theta^b_{3r} - T^r_{13}\theta^b_{r1} \big) = T^1_{13}\theta^b_{33},$$
hence $\theta^b_{33}=0$. In other words, the fixing of the gauge $T^1_{13}=|T^1_{13}|>0$ made our admissible frame $e$ to satisfy $\nabla^be_3=0$. Similarly, by $T^2_{13}=0$, we have
$$ 0 =  dT^2_{13} = \sum_r \big( T^2_{r3}\theta^b_{1r} + T^2_{1r}\theta^b_{3r} - T^r_{13}\theta^b_{r2} \big) = ( T^2_{23} - T^1_{13})\theta^b_{12} = -2T^1_{13} \theta^b_{12} ,$$
which implies that $\theta^b_{12}=0$, therefore  $\theta^b$ is diagonal.  Hence by the {\bf Claim} we get $R^b=0$. This completes the proof of Theorem \ref{prop6.11}, which is Theorem \ref{prop1.7a}.
\end{proof}

The above proof of the $n=4$ case of Conjecture \ref{conj6.1} is basically by brute force, which relies heavily on $n$ being small. The proof in fact implies the following

\begin{corollary} \label{coro-nb}
Suppose $(M^n,g)$ is a BTP manifold. Assume that the Bismut scalar curvatures satisfy $S^{b(1)}=S^{b(3)}$. Then $(M^n, g)$ is not K\"ahler if and only if it is not balanced.
\end{corollary}
\vspace{0.1cm}

Recall that a Hermitian manifold is said to be {\em Bismut K\"ahler-like} (or {\em BKL} for short, see \cite{AOUV} and \cite{YZ}), if the Bismut curvature tensor $R^b$ obeys all the K\"ahler symmetries: $R^b_{xyz\bar{w}}=0$ and  $R^b_{x\bar{y}z\bar{w}}= R^b_{z\bar{y}x\bar{w}}$ for any type $(1,0)$ tangent vectors $x,y,z,w$. When $n\geq 2$,  there are examples of non-K\"ahler manifolds which are Bismut K\"ahler-like, and
such metrics were classified in complex dimensions $n \leq 5$ (\cite{YZZ, ZZpluriclosed, ZZ23}). By the main result of \cite{ZZpluriclosed}, the BKL condition is equivalent to BTP plus pluriclosedness. (On the other hand, BKL and BAS are two sets that neither one is contained in the other, although they are both contained in the BTP set).

Now let us prove the following proposition stated in the introduction (= Proposition \ref{prop1.8a}), which is a special case of \cite[Theorem 3]{ZZ23}. Here we give an alternative proof as an application of Proposition \ref{prop:32} (or the more general version Theorem \ref{thm:generlizedholo}):

\begin{proposition} \label{prop6.7}
Let $(M^n,g)$ be a complete Hermitian manifold whose Bismut connection is Ambrose-Singer (BAS) and with vanishing first Bismut Ricci curvature. If $g$ is BKL, then $g$ is Bismut flat.
\end{proposition}

\begin{proof}
Fix any $p\in M$ and let $V$ be the tangent space of $M$ at $p$. Note that the Bismut K\"ahler-like assumption guarantees that $\Rm^b$ obeys the first Bianchi identity, so the Bismut connection $\nabla^b$ gives a holonomy system, which is clearly a symmetric one. In order to apply Proposition \ref{prop:32}, we need to verify that that the Ricci curvature vanishes for each irreducible component. Decompose the tangent bundle into subbundles where the Bismut holonomy group acts irreducibly. Writing in complex frames, say $T^{1,0}M=\oplus_{j=1}^r E_j$, then the Bismut curvature form $\Theta^b$ is block-diagonal:
$$ \Theta^b = \left[ \begin{array}{ccc} \Theta^b_1 & & \\ & \ddots & \\ & & \Theta^b_r \end{array} \right]  , \ \ \Theta^b_j = (\Theta^b_{ik}), \ \ \ e_i, e_k \in E_j.$$
Since $\nabla^b$ is assumed to be K\"ahler-like, the entries of $\Theta^b_j$ are all combinations  of $\varphi_i\overline{\varphi}_k$ for $e_1, e_k \in E_j$ only. In particular, $\mbox{tr}(\Theta^b) =0$ if and only if $\mbox{tr}(\Theta^b_j) =0$ for each $j$. So if we assume that the Bismut curvature has vanishing Ricci, then each irreducible component will also have vanishing Ricci, thus one can apply Proposition \ref{prop:32} to conclude that the curvature vanishes. This completes the proof. Note that since the curvature is assumed to be parallel, so any K\"ahler de Rham factor will be locally Hermitian symmetric, hence it will be flat if it is Ricci flat.
\end{proof}

Recall that compact Bismut flat manifolds were classified \cite{WYZ}. They are quotients of Samelson spaces, namely, Lie groups with bi-invariant metrics and compatible left-invariant complex structures. Hermitian metrics with vanishing first Bismut Ricci are said to be {\em Calabi-Yau with torsion} (or {\em CYT} in short) in the literature. It is equivalent to the condition when the restricted holonomy group of $\nabla^b$ is contained in $SU(n)$.

\vspace{0.1cm}

An interesting question closely related to Conjecture \ref{conj6.1} is the following one raised by Garcia-Fern\'andez and Streets \cite[Question 3.58]{GFS}:

\begin{question}[Garcia-Fern\'andez and Streets] \label{question6.8}
Given a homogeneous  Riemannian manifold $(M,g)$ with an invariant $3$-form $H$, denote by $\nabla^H$ the metric connection  with totally skew-symmetric torsion $H$. If $\nabla^H$ has vanishing Ricci, then must it be flat?
\end{question}

Such a $(g,H)$ is called a {\em Bismut Ricci Flat} pair, or a {\em BRF pair} in short. This happens when and only when $H$ is harmonic and the Ricci curvature $\mbox{Ric}^g$ of the Levi-Civita connection $\nabla^g$ is given by
$$ \mbox{Ric}^g (x,y) = \frac{1}{4} \langle \iota_xH, \iota_yH \rangle $$
for any tangent vector $x$, $y$, where $\iota_x$ stands for contraction. This concept is important in generalized Riemannian geometry which corresponds to special types of generalized Einstein structures.

Garcia-Fern\'andez and Streets answered the question positively for real dimension $3$ in (a special case of) \cite[Proposition 3.55]{GFS}.
In \cite{Podesta-R1} and \cite{Podesta-R}, Podest\`a and Raffero investigated the question and answered it positively in  (real) dimension $4$, but in dimension $5$ or higher, they constructed counterexamples. To be more precise, in \cite{Podesta-R1} they constructed an explicit sequence $M_{p,q}$ of compact homogeneous $5$-manifolds with invariant closed $3$-form $H$ so that $\nabla^H$ is Ricci flat but not flat. Here $p$ and $q$ are any positive integers that are relatively prime.

Note that Question \ref{question6.8} is closely related to Conjecture \ref{conj6.1} but different. The latter assumes the manifold to be Hermitian and the Bismut connection to be Ambrose-Singer, but it does not require the torsion $3$-form to be closed. It would be interesting to understand all BRF pairs, and in particular all BRF pairs where $\nabla^H$ is also Ambrose-Singer.

\vspace{0.3cm}

\section{Gauduchon Ambrose-Singer manifolds}

In this section we consider $t$-GAS manifolds: compact Hermitian manifolds whose $t$-Gauduchon connection $\nabla^{(t)}$ is Ambrose-Singer (i.e., having parallel torsion and curvature). Assume $t\neq 0,2$.   First we note that there are such manifolds that are non-K\"ahler (hence not $\nabla^{(t)}$-flat by the theorem of Lafuente and Stanfield \cite{LS}). The simplest such example would be a compact quotient  of a complex semi-simple Lie group, equipped with the metric coming from the Cartan-Killing form (cf. \S 3 of \cite{PodestaZ}). Such a manifold is actually $t$-GAS for any $t\in {\mathbb R}$.

To be more precise, let $G$ be a simple complex Lie group and denote by ${\mathfrak g}$ its Lie algebra. Let us write ${\mathfrak g}_{\mathbb R}$ for the Lie algebra ${\mathfrak g}$ considered as a real Lie algebra. Then the Cartan decomposition of ${\mathfrak g}_{\mathbb R}$ is given by ${\mathfrak g}_{\mathbb R} ={\mathfrak u} + J {\mathfrak u}$, where ${\mathfrak u}$ is the Lie algebra of a compact simple Lie group $U$. If we denote by $K$ the Cartan-Killing form of ${\mathfrak g}_{\mathbb R}$, then the `metric coming from the Cartan-Killing form' is the one given by
$$ g|_{{\mathfrak u} \times {\mathfrak u}} = -K, \ \ \ g({\mathfrak u}, J{\mathfrak u})=0, \ \ \ g|_{J{\mathfrak u} \times J{\mathfrak u}} = K. $$
By the proof of Theorem 1.2 of \cite{PodestaZ}, we know that $(G,g)$ is BTP. It is clearly also Chern flat.  By the first equality in Lemma 7.2 below, the BTP condition would mean that the equality holds for $t=2$, thus the big sigma term on the right hand side would vanish. So the same equality again would imply that $\nabla^{(t)}_{\overline{e}_i}T=0$. Similarly, one could argue and get  $\nabla^{(t)}_{e_i}T=0$, so $\nabla^{(t)}T=0$ (which is equivalent to $\nabla^{(t)}T^{(t)}=0$), namely, the metric is also $t$-Gauduchon torsion parallel. By the second equality in Lemma 7.2, we know that $\nabla^{(t)}R^{(t)}=0$. So $(G,g)$ is $t$-GAS for all $t$.

Secondly, we notice that for any $t\neq 2$, a compact $t$-GAS manifold is always balanced:

\begin{proposition} \label{lemma6.9}
Let $(M^n,g)$ be a compact Hermitian manifold that is $t$-GAS for some $t\neq 2$. Then $g$ is balanced.
\end{proposition}

\begin{proof}
Under any local unitary frame $e$, let us denote by $\theta^{(t)}$ the matrix of the $t$-Gauduchon connection $\nabla^{(t)}$, and write  $\theta^{(0)}=\theta$ for the Chern connection. Then $\theta^{(t)} = \theta + \frac{t}{2} \gamma $ where $\gamma'_{ij}=\sum_k T^j_{ik}\varphi_k$ and $\gamma = \gamma' - \,^t\!\overline{\gamma}'$. Here as before we denoted by $\varphi$ the coframe dual to $e$, and by $T^j_{ik}$ the components of the Chern torsion under $e$, namely, $T(e_i,e_k)=\sum_jT^j_{ik}e_j$. Denote by $\eta = \sum_i \eta_i \varphi_i$ the Gauduchon torsion $1$-form, defined by $\eta_i=\sum_k T^k_{ki}$. Then one has $\partial (\omega^{n-1})=-\eta \wedge \omega^{n-1}$, where $\omega =\sqrt{-1}\sum_k \varphi_k \wedge \overline{\varphi}_k$ is the K\"ahler form of $g$.

The structure equation is $d\varphi = -\,^t\!\theta \wedge \varphi + \tau$, where $\tau = \frac{1}{2}\,^t\!\gamma' \wedge \varphi$ is the column vector of Chern torsion. For any fixed $p\in M$, let us choose our local unitary frame  $e$ so that $\theta^{(t)}$ vanishes at $p$. Then at $p$ we have $\theta = -\frac{t}{2}\gamma$, thus by the structure equation we know that $\overline{\partial} \varphi_r=  \frac{t}{2} \sum_{i,j}\overline{T^j_{ri}} \,\varphi_i \wedge \overline{\varphi}_j$ at $p$. We compute that
$$ \overline{\partial}\eta = \sum_{i,j} \{ - \eta_{i,\overline{j}} +  \frac{t}{2}\sum_r \eta_r \overline{T^i_{rj}} \} \,\varphi_i \wedge \overline{\varphi}_j , $$
where the index after comma stands for covariant derivative with respect to $\nabla^{(t)}$. From this we get
$$ n\sqrt{-1} \,\overline{\partial}\eta \wedge \omega^{n-1} = - \big(  \sum_i \eta_{i,\overline{i}} + \frac{t}{2}|\eta|^2\big) \omega^n, $$
where $|\eta|^2 = \sum_i |\eta_i|^2$. Similarly, $n \sqrt{-1} \, \eta \wedge \overline{\eta}\wedge \omega^{n-1} = |\eta|^2 \omega^n$. On the other hand,
$$ \partial \overline{\partial} (\omega^{n-1}) = \overline{\partial} (\eta \wedge \omega^{n-1}) = (\overline{\partial}\eta + \eta \wedge \overline{\eta})\wedge \omega^{n-1} = \frac{\sqrt{-1}}{n} \{ \sum_i \eta_{i,\overline{i}} + (\frac{t}{2}-1)|\eta |^2 \} \omega^n.  $$
When the metric is $t$-GAS, we have $\nabla^{(t)}T=0$ hence $\eta_{i,\overline{j}}=0$. Integrating the last equation above over the compact manifold $M$, we get $(\frac{t}{2}-1) \int_M |\eta|^2\omega^n =0$. Hence for any $t\neq 2$ we would have $\eta=0$, which means that $g$ is balanced.
\end{proof}

We remark that the above result also was proved in \cite{NZ-CAS}, namely for the $t=0$ case. It indicates that the Bismut connection is a bit special in view of Theorem \ref{prop6.11}. Denote by  $T$, $R$ the torsion and curvature of the Chern connection $\nabla$, and the components of $T$ under a unitary frame $e$ are given by $T(e_i,e_k)=\sum_j T^j_{ik}e_j$. Denote by $R^{(t)}$ the curvature of the $t$-Gauduchon connection $\nabla^{(t)}$. We have the following

\begin{lemma}
Let $(M^n,g)$ be a Hermitian manifold. Then under any local unitary frame $e$ it holds:
\begin{eqnarray}
&&T^{\ell}_{ik,\bar{j}} + R_{i\bar{j}k\bar{\ell}} - R_{k\bar{j}i\bar{\ell}} +\frac{t}{2} \sum_r \{ T^{\ell}_{ir} \overline{T^{k}_{jr} } - T^{\ell}_{kr} \overline{T^{i}_{jr} }  - T^{r}_{ik} \overline{T^{r}_{j\ell } }  \} \ = \ 0,     \label{eq:6.11}\\
&& R^{(t)}_{i\bar{j}k\bar{\ell}} - R_{i\bar{j}k\bar{\ell}} - \frac{t}{2} \{ T^{\ell}_{ik,\bar{j} } + \overline{T^{k}_{j\ell , \bar{i} } }\}  +  \frac{t^2}{4} \sum_r  \{ T^{j}_{ir} \overline{T^{k}_{\ell r} } + T^{\ell}_{kr} \overline{T^{i}_{jr} } + T^{r}_{ik} \overline{T^{r}_{j\ell } } - T^{\ell}_{ir} \overline{T^{k}_{jr} } \} \ = \ 0  \label{eq:6.12}
\end{eqnarray}
for any $1\leq i,j,k,\ell \leq n$, where index after comma stands for covariant derivative with respect to $\nabla^{(t)}$.
\end{lemma}

\begin{proof} Let $\theta$, $\theta^{(t)}$, $\gamma$ , $\tau$ be as in the proof of Lemma \ref{lemma6.9}. Denote by $\Theta$, $\Theta^{(t)}$ the matrix of Chern curvature and $t$-Gauduchon curvature, respectively. We have $\theta^{(t)}=\theta + \frac{t}{2}\gamma$, and $^t\!\gamma' \wedge \varphi = 2\tau$. Take the exterior differential on the structure equation $d\varphi = -\,^t\!\theta \wedge \varphi + \tau$, we get the first Bianchi identity:
\begin{equation}
d\tau = - \,^t\!\theta \wedge \tau + \,^t\!\Theta \wedge \varphi.    \label{eq:Bianchi}
\end{equation}
Fix any point $p\in M$. Choose the local unitary frame $e$ near $p$ so that $\theta^{(t)}|_p=0$. Then at $p$ we have $\theta = -\frac{t}{2}\gamma$, hence
 \begin{equation}  \label{eq:7.4}
 \partial \varphi =(1-t)\tau = \frac{1}{2}(1-t) \,^t\!\gamma' \wedge \varphi \ \ \ \mbox{and} \ \ \ \overline{\partial} \varphi = -\frac{t}{2} \overline{\gamma}' \wedge \varphi \ \ \ \ \mbox{at} \ p.
 \end{equation}
Now take the $(2,1)$-part in (\ref{eq:Bianchi}), at $p$ we have
$$ \sum_k \{ \overline{\partial}   \gamma'_{k\ell} + \frac{t}{2} \sum_r \big( \overline{\gamma}'_{\ell r} \wedge \gamma'_{kr}  + \gamma'_{r\ell } \wedge \overline{\gamma}'_{rk}  \big) + 2\Theta_{k\ell} \} \wedge \varphi_k \, = \, 0
$$
for any $1\leq \ell \leq n$. This leads us to the identity (\ref{eq:6.11}). On the other hand, by definition we have  $\Theta = d\theta - \theta \wedge \theta$  and  $\Theta^{(t)} = d\theta^{(t)} - \theta^{(t)} \wedge \theta^{(t)}$. So at $p$ we get
$$ (\Theta^{(t)})^{1,1} - \Theta \ = \ \frac{t}{2}\{  \overline{\partial}   \gamma'- \partial \,^t\!\overline{\gamma}'\} -\frac{t^2}{4} \{ \gamma' \wedge \,^t\!\overline{\gamma}' + \,^t\!\overline{\gamma}' \wedge \gamma' \} .$$
This leads to the identity (\ref{eq:6.12}), and the lemma is proved.
 \end{proof}

\begin{proof}[{\bf Proof of Theorem \ref{prop1.10}}.]
Suppose $t\neq 0,2$ and $(M^n,g)$ is a compact $t$-GAS manifold with vanishing first and third $t$-Gauduchon Ricci. By Lemma \ref{lemma6.9} we know that $g$ is balanced. For $1\leq i\leq 3$, denote by $\Ric^{(t)(i)}$ the $i$-th Ricci of $\nabla^{(t)}$, and by $\Ric^{(i)}$ the $i$-th Chern Ricci. Since $\eta =0$ and  $\Ric^{(t)(1)}=0$, by letting $k=\ell$ and sum up in (\ref{eq:6.12}), we get $\Ric^{(1)}=0$. Similarly, since $\eta=0$ and $\Ric^{(t)(3)}=0$,  by letting $i=\ell$ and sum up in (\ref{eq:6.12}), we get $\Ric^{(3)}=\frac{t^2}{4} B$, where $B_{k\bar{j}} = \sum_{r,s} T^j_{rs} \overline{T^k_{rs} } $. Now if we let $k=\ell$ and sum up in (\ref{eq:6.11}), we get $\Ric^{(1)}-\Ric^{(3)}=0$. Putting these together, we get $B=0$, hence $\mbox{tr}(B) = |T|^2 =0$. Therefore $T=0$ and $g$ is K\"ahler. This completes the proof of Theorem \ref{prop1.10}.
\end{proof}

  We remark that if we drop the assumption of  $\Ric^{(t)(1)}=0$ in Theorem \ref{prop1.10}, then we only get $\Ric^{(1)}=\frac{t^2}{4}B$, and we have

\begin{proposition}
 Let $(M^n,g)$ be a $t$-GAS manifold with $t\neq 0,2$, and assume that the third $t$-Gauduchon Ricci vanishes, namely, $\Ric^{(t)(3)}=0$. Then $g$ is balanced and $\Ric^{(1)}=\frac{t^2}{4}B$. If $t<0$, then $g$ must be K\"ahler and flat. If $t>0$, then $M$ cannot be Fano.
 \end{proposition}

\begin{proof} By the condition $\Ric^{(1)}=\frac{t^2}{4}B$ we know that the global $(1,1)$-form $\sigma_B = \sqrt{-1}\sum_{i,j} B_{i\bar{j}} \varphi_i \wedge \overline{\varphi}_j$ is equal to a positive constant multiple of the first Chern Ricci form, hence is $d$-closed. Since $\nabla^{(t)}B=0$, the eigenvalues of $B$ are all constants. Choose $e$ so that $B$ is diagonal, then by $d\sigma_B=0$ we get from (\ref{eq:7.4}) that:
 \begin{equation} \label{eq:7.5}
 \{ tB_{i\bar{i}} + tB_{k\bar{k}} -2(t-1)B_{j\bar{j}} \} \,T^j_{ik} =0 , \ \ \ \ \ \ \ \ \forall \ 1\leq i,j,k\leq n.
\end{equation}
Write $\alpha = \frac{2(t-1)}{t} = 2(1 - \frac{1}{t})$. When $t<0$, we have $\alpha >2$. Let us choose $e$ so that $B$ is diagonal and $B_{1\bar{1}}$ is the largest eigenvalue. If $B_{1\bar{1}} >0$, then
$$ \alpha B_{1\bar{1}} - B_{i\bar{i}} - B_{i\bar{i}} \geq (\alpha -2)B_{1\bar{1}} > 0, $$
hence by (\ref{eq:7.5}) we get $T^1_{ik}=0$ for any $i,k$. Hence $B_{1\bar{1}}=0$ by the definition of $B$. This shows that $B_{1\bar{1}}=0$ hence $B=0$, so $g$ is K\"ahler, hence flat as it is assumed to be $t$-GAS.
If $t>0$, then $\alpha < 2$. Let us choose $e$ so that $B$ is diagonal and $B_{1\bar{1}}$ is the smallest eigenvalue. If $B_{1\bar{1}} >0$, then
$$ \alpha B_{1\bar{1}} - B_{i\bar{i}} - B_{i\bar{i}} \leq (\alpha -2)B_{1\bar{1}} < 0, $$
so by (\ref{eq:7.5}) we again get $T^1_{ik}=0$ for any $i,k$, which means $B_{1\bar{1}}=0$. So the smallest eigenvalue of $B$, hence $\Ric^{(1)}$, must be zero, which means that $M$ (has nef ant-canonical line bundle but) is not Fano.
\end{proof}

While this is certainly highly restrictive, but without the assumption on the vanishing of the first $t$-Gauduchon Ricci, we do not know how to get the flatness/K\"ahlerity of $(M^n, g)$. One could certainly replace that by the vanishing of the second $t$-Gauduchon Ricci, or simply the vanishing of the trace of either of them, namely, the vanishing of the first (= second) $t$-Gauduchon scalar curvature.

\begin{corollary}\label{coro:71} Let $(M^n, g)$ be a compact t-GAS manifold with $t\ne 2$. Then
$\,\Ric^{(t)(1)}=\Ric^{(1)}=\Ric^{(3)}$ and $\,\Ric^{(t)(3)}-\Ric^{(3)}=-\frac{t^2}{4}B$. So the scalar curvatures satisfy  $\,S^{(t)(1)}=S^{(1)}=S^{(3)}$ and  $\,S^{(t)(3)}-S^{(3)}=-\frac{t^2}{4}|T|^2$. In particular, the vanishing of $S^{(t)(1)}$ and $S^{(t)(3)}$ would imply that $(M^n, g)$ is K\"ahler and flat.
\end{corollary}

\vspace{0.3cm}

\section{Appendix: The Hopf manifold is BAS}

In this appendix, we give the calculation showing that the standard Hopf manifold $(M^n,g)$ is BAS, where $M^n=({\mathbb C}^n\setminus \{0\})/ \langle f\rangle$, $f(z)=2z$, $z=(z_1, \ldots , z_n)$, $g_{i\bar{j}} =\frac{1}{|z|^2}\delta_{ij}$ where $|z|^2=|z_1|^2 +\cdots + |z_n|^2$.

Let $e$ be the local unitary frame with $e_i = |z| \frac{\partial}{\partial z_i}$. Its dual coframe  $\varphi$ is given by $\varphi_i = \frac{1}{|z|}dz_i$. The matrix of the metric under the local holomorphic coordinate $z$ is $g=\frac{1}{|z|^2}I$.  Write $P=|z|I$ for the frame change matrix, then the Chern connection matrix under $e$ is given by
$$ \theta = P(\partial g g^{-1})P^{-1} + dP P^{-1} = (-\partial + \frac{1}{2}d)\log |z|^2 \,I .$$
Write $\nabla_{e_k}e_i=\sum_j \theta_{ij}(e_k)e_j  = \Gamma^j_{ik} e_j$, then the connection coefficients are $\Gamma^j_{ik}= -\frac{\overline{z}_k}{2|z|}\delta_{ij}$. On the other hand, by the structure equation, we have $\tau = d\varphi + \,^t\!\theta \wedge \varphi$, so the Chern torsion components under $e$ are
\begin{equation*} \label{Aeq:torsion}
 T^j_{ik} =  \frac{\overline{z}_k}{|z|} \delta_{ji} -\frac{\overline{z}_i}{|z|} \delta_{jk} .
  \end{equation*}
Therefore the Bismut connection coefficients under $e$ are
$$\Gamma'^{j}_{ik}= \Gamma^j_{ik} + T^j_{ik} =  \frac{\overline{z}_k}{2|z|} \delta_{ji} -\frac{\overline{z}_i}{|z|} \delta_{jk} , $$
and the Bismut connection matrix under $e$ are
\begin{equation} \label{Aeq:thetab}
 \theta^b_{ij} = \sum_k \{ \Gamma'^j_{ik}\varphi_k  - \overline{\Gamma'^i_{jk}} \overline{\varphi}_k  \}  = \frac{1}{2}(\overline{\partial} -\partial )\log |z|^2 \,\delta_{ij}  + \frac{1}{|z|^2} ( d\overline{z}_i z_j - \overline{z}_i dz_j ).
 \end{equation}
Write $ \Theta^b = d\theta^b -\theta^b \wedge \theta^b$, then by a direct computation we find that the Bismut curvature components $R^b_{i\bar{j}k\bar{\ell}} = \Theta^b_{k\ell } (e_j , \overline{e}_j)$ under $e$ are:
\begin{equation} \label{Aeq:Rb}
 R^b_{i\bar{j}k\bar{\ell}} = \delta_{i\ell} \delta_{kj} - \delta_{ij} \delta_{k\ell} + \frac{1}{|z|^2} \{ \overline{z}_iz_j \delta_{k\ell} + \overline{z}_kz_{\ell} \delta_{ij} -  \overline{z}_iz_{\ell} \delta_{kj} -  \overline{z}_kz_j \delta_{i\ell} \} .
 \end{equation}

{\bf Claim 1:} The Hopf manifold $(M^n,g)$ has Bismut parallel torsion (namely, BTP): $\nabla^b T^b=0$.

It is well known that $(M^n,g)$ is Vaisman, namely, locally conformally K\"ahler with its Lee form being parallel under the Levi-Civita connection $\nabla^g$. In fact, the K\"ahler form of $g$ is
$$\omega = \sqrt{-1} \frac{\partial \overline{\partial} |z|^2}{|z|^2},$$
so $d\omega = \psi\wedge \omega$ with $\psi = -d\log |z|^2$, which is a global closed $1$-form on $M^n$. Hence $g$ is locally conformally K\"ahler with Lee form $\psi$. It is easy to check that $\nabla^g \psi =0$, so $g$ is Vaisman. By \cite[Corollary 3.8]{AndV} we know that $g$ is BTP.

Alternatively, one could also directly verify that $g$ is BTP, namely, $\nabla^bT^b=0$, which is equivalent to $\nabla^bT=0$ as the Bismut torsion $T^b$ and Chern torsion $T$ can be expressed in terms of each other. $\nabla^bT=0$ means that
\begin{eqnarray*}
&&   e_s(T^j_{ik}) - \sum_r \{ T^j_{rk} \Gamma'^r_{is} +  T^j_{ir} \Gamma'^r_{ks} - T^r_{ik} \Gamma'^j_{rs} \} \, = \,0, \\
&& \overline{e_s}(T^j_{ik}) +  \sum_r \{ T^j_{rk} \overline{\Gamma'^i_{rs}} + T^j_{ir} \overline{\Gamma'^k_{rs}}  - T^r_{ik} \overline{\Gamma'^r_{js}}  \} \, = \, 0 ,
\end{eqnarray*}
for any $1\leq i.j,k,s\leq n$. By the expressions for $T^j_{ik}$ and $\Gamma'^j_{ik}$ before, one can verify the above directly. This establishes the proof of Claim 1. \qed

\vspace{0.1cm}

 {\bf Claim 2:} The Hopf manifold $(M^n,g)$ is not Bismut flat when $n\geq 3$.

To see this, note that the trace of $\theta^b\wedge \theta^b$ is zero, so by (\ref{Aeq:thetab}) we get
$$ \mbox{tr}(\Theta^b)=d \,\mbox{tr}(\theta^b) = \frac{n-2}{2}\partial \overline{\partial} \log |z|^2. $$
So when $n\geq 3$, the first Bismut Ricci curvature of the Hopf manifold $(M^n,g)$ does not vanish identically, thus it is not Bismut flat. \qed

\vspace{0.1cm}
We remark that Claim 2 also follows from Theorem 5.9 of \cite{AndV}, which asserts that the holonomy of the above  Hopf manifold is $\mathsf{U}(n-1)$ when $n\ge 3$.

 {\bf Claim 3:} The Hopf manifold $(M^n,g)$ has parallel Bismut curvature: $\nabla^bR^b=0$.

 This means that for any indices $1\leq i,j,k,\ell, s\leq n$, it holds that
 \begin{equation} \label{Aeq:nablaRb}
 e_s(R^b_{i\bar{j}k\bar{\ell}}) = \sum_r \{ R^b_{r\bar{j}k\bar{\ell}}  \Gamma'^r_{is} -  R^b_{i\bar{r}k\bar{\ell}}  \Gamma'^j_{rs}  +   R^b_{i\bar{j}r\bar{\ell}}  \Gamma'^r_{ks} -  R^b_{i\bar{j}k\bar{r}}  \Gamma'^{\ell}_{rs}    \} .
 \end{equation}
By (\ref{Aeq:Rb}), we notice that
\begin{equation*}
 R^b_{i\bar{j}k\bar{\ell}} = - R^b_{k\bar{j}i\bar{\ell}}, \ \ \ \ \forall \ 1\leq i,j,k,\ell \leq n.
 \end{equation*}
So a component $R^b_{i\bar{j}k\bar{\ell}}$ can be non-trivial only if $i\neq k$ and $j\neq \ell$. If $i=k$, then the second and fourth terms in the right hand side of (\ref{Aeq:nablaRb}) vanish, while the first and third term add up to zero, so (\ref{Aeq:nablaRb}) holds. Similarly, the equality holds when $j=\ell$, thus in the following we will assume that $i\neq k$ and $j\neq \ell$.

From (\ref{Aeq:Rb}), we compute that
\begin{equation} \label{Aeq:1}
e_s(R^b_{i\bar{j}k\bar{\ell}} ) = \big(\frac{\overline{z}_i}{|z|} \delta_{k\ell} -  \frac{\overline{z}_k}{|z|} \delta_{i\ell} \big) \big( \delta_{js} - \frac{\overline{z}_s z_j }{|z|^2} \big) + \big(\frac{\overline{z}_k}{|z|} \delta_{ij} -  \frac{\overline{z}_i}{|z|} \delta_{kj} \big) \big( \delta_{\ell s} - \frac{\overline{z}_s z_{\ell} }{|z|^2} \big).
\end{equation}

Next we notice that  the unitary group $U(n)$ acts holomorphically and isometrically on the universal cover of $(M^n,g)$, so we just need to verify (\ref{Aeq:nablaRb})  at the special points $p=(\lambda , 0 ,\ldots , 0)$, where $\lambda >0$ is an arbitrary constant. At $p$, the Bismut curvature components are
$$ R^b_{i\bar{j}k\bar{\ell}} (p) = \delta_{i\ell} \delta_{kj} - \delta_{ij} \delta_{k\ell} + \{ \delta_{ij1} \delta_{k\ell} + \delta_{k\ell 1} \delta_{ij} - \delta_{i\ell 1} \delta_{kj} - \delta_{kj1} \delta_{i\ell} \}, $$
where we wrote $\delta_{ij1}$ for $\delta_{ij}\delta_{i1}$. Therefore $R^b_{i\bar{j}k\bar{\ell}} (p)=0$ when $\{i,k\}\neq \{ j,\ell\}$, while
$$R^b_{i\bar{i}k\bar{k}} (p) = - R^b_{k\bar{i}i\bar{k}} (p) = -1 +\delta_{i1}+\delta_{k1}, \ \ \ \ \forall \  i\neq k. $$
By (\ref{Aeq:1}), we know that the value of $e_s(R^b_{i\bar{j}k\bar{\ell}} )$ at $p$ is equal to
$$ \big( e_s(R^b_{i\bar{j}k\bar{\ell}} ) \big) (p) = \big( \delta_{i1}\delta_{k\ell} - \delta_{k1}\delta_{i\ell } \big) \big( \delta_{js} -\delta_{js1}\big) +   \big( \delta_{i1}\delta_{kj} - \delta_{k1}\delta_{ij} \big) \big( \delta_{\ell s} -\delta_{\ell s1}\big)  . $$
Note that this quantity is zero when $i,k>1$, and when $i=1<k$ it equals to
$$\delta_{k\ell }  \big( \delta_{js} -\delta_{js1}\big) + \delta_{kj}  \big( \delta_{\ell s} -\delta_{\ell s1}\big) .$$

Now let us assume that $j\notin \{ i,k\}$. Then the left hand side of (\ref{Aeq:nablaRb}) is
$$  L := \big( \delta_{i1}\delta_{k\ell} - \delta_{k1}\delta_{i\ell } \big) \big( \delta_{js} -\delta_{js1}\big) ,$$
while the right hand side of (\ref{Aeq:nablaRb}) is
$$  R:= R^b_{j\bar{j}k\bar{\ell}}  \Gamma'^j_{is} -  \big(  R^b_{i\bar{i}k\bar{\ell}}  \Gamma'^j_{is}  + R^b_{i\bar{k}k\bar{\ell}}  \Gamma'^j_{ks} \big) +   R^b_{i\bar{j}j\bar{\ell}}  \Gamma'^j_{ks}  = (R^b_{j\bar{j}k\bar{\ell}}  - R^b_{i\bar{i}k\bar{\ell}}  )  \Gamma'^j_{is}  + ( R^b_{k\bar{k}i\bar{\ell}}  - R^b_{j\bar{j}i\bar{\ell}}  )  \Gamma'^j_{ks}. $$
If $\ell \notin \{ i,k\}$, then both $L$ and $R$ are zero, thus equal. If $\ell =i$, then $L=-\delta_{k1}\big( \delta_{js}-\delta_{js1}\big)$, while
$$ R=  \{ R^b_{k\bar{k}i\bar{i}}  - R^b_{j\bar{j}i\bar{i}}  \} \Gamma'^j_{ks} = \{ (-1+\delta_{k1}+\delta_{i1}) - (-1+\delta_{j1}+\delta_{i1}) \} (-\delta_{js}\delta_{k1}) = \big( \delta_{j1} -\delta_{k1} \big) \delta_{js} \delta_{k1}.$$
When $k>1$, both are zero, while when $k=1$, both are equal to $\delta_{js1}-\delta_{js}$. Similarly, if $\ell =k$, then we have $L=R$.

We are now left with the case $\{ i,k\} = \{ j, \ell \}$ and $i\neq k$. Without loss of generality, let us assume that $i=k\neq j=\ell$. In this case, the left hand side of (\ref{Aeq:nablaRb}) is
$$ \big( e_s(R^b_{i\bar{i}k\bar{k}} ) \big) (p) = \delta_{i1}(\delta_{is} -\delta_{is1}) - \delta_{k1} (\delta_{ks} - \delta_{ks1}) = 0, $$
and the right hand side of (\ref{Aeq:nablaRb}) is
$$ \sum_r \{ R^b_{r\bar{i}k\bar{k}}  \Gamma'^r_{is} -  R^b_{i\bar{r}k\bar{k}}  \Gamma'^i_{rs}  +   R^b_{i\bar{i}r\bar{k}}  \Gamma'^r_{ks} -  R^b_{i\bar{i}k\bar{r}}  \Gamma'^{k}_{rs}    \}     =        R^b_{i\bar{i}k\bar{k}}  \Gamma'^i_{is} -  R^b_{i\bar{i}k\bar{k}}  \Gamma'^i_{is}  +   R^b_{i\bar{i}k\bar{k}}  \Gamma'^k_{ks} -  R^b_{i\bar{i}k\bar{k}}  \Gamma'^{k}_{ks}  = 0 .$$
So in all cases the equality (\ref{Aeq:nablaRb}) holds. This completes the proof of Claim 3. \qed

\vspace{0.1cm}

In summary, we have shown that the Hopf manifold $(M^n,g)$ is always BAS, and it is not Bismut flat when $n\geq 3$.

\bigskip

\bigskip

\noindent\textbf{Acknowledgments.} {We would like to thank C. B\"ohm for his interests and bringing our attention to a result of Alekseevski\u{i} and B. N.  Kimel\'{}fel\'{}d,  P. Petersen and H. Wu for their enthusiastic interests and encouragement,  J. Stanfield and S. Gindi for their interests in extending the result on Chern connection to the BAS manifolds, Y. Ustinovskiy for his interest in extending results to the torsion connection. The first author would like to thank T. Colding and MIT mathematics for the hospitality during his visit when the first draft of the paper was completed, B. Kleiner for insightful questions,  D. Yang and G. Zhang for their hospitality during his visit of CIMS at NYU, and SUSTech for its hospitality at the final stage of the writing. We are also grateful to J. Streets for his interests on the results related to the Bismut connection and for the references \cite{GFS}, \cite{Podesta-R}, A. Raffero for correcting some inaccurate referencing. Finally we also thank R. Lafeunte and a referee for helpful comments on various presentation improvements.}

\end{document}